\documentclass[dvipdfmx]{article}

\usepackage[top=40truemm,bottom=40truemm,left=38truemm,right=43truemm]{geometry}

\usepackage{indentfirst}

\usepackage[all,cmtip]{xy}
\usepackage{pb-diagram}
\usepackage{hyperref}
\usepackage{amsmath,amsxtra,amssymb,mathrsfs,verbatim,amscd}
\usepackage{amsfonts,stmaryrd}
\usepackage{dsfont}
\usepackage{comment}
\usepackage{ascmac}
\usepackage{amsthm}
\usepackage{color}
\usepackage{graphicx}
\usepackage{tikz}
\usetikzlibrary{positioning,intersections,math}
\usepackage{accents}
\usepackage{here}

\allowdisplaybreaks

\makeatletter

\@addtoreset{equation}{section}
\makeatother

\theoremstyle{definition}
\newtheorem{theo}{Theorem}[section]
\newtheorem{lemm}[theo]{Lemma}
\newtheorem{defi}[theo]{Definition}
\newtheorem{prop}[theo]{Proposition}
\newtheorem{coro}[theo]{Corollary}
\newtheorem{rema}[theo]{Remark}
\newtheorem{exam}[theo]{Example}

\newcommand{\ilim}[1][]{\mathop{\varinjlim}\limits_{#1}}

\newcommand{\psum}[1][]{\mathop{\textstyle\sum'}_{#1}}

\newcommand{\Cinf}[2]{C^{\infty,({#1,#2})}_X}
\newcommand{\shE}{\mathscr{E}^{\mathbb{R}}_{X}}
\newcommand{\shEx}{\mathscr{E}^{\mathbb{R}}_{X,z^*}}

\newcommand{\bp}{\bar{\partial}}
\newcommand{\Q}{\mathfrak{S}/\mathfrak{N}}
\newcommand{\Qinf}{\mathfrak{S}^\infty/\mathfrak{N}^\infty}
\newcommand{\cSubset}{\underset{\rm cone}{\Subset}}

\newcommand{\wsigma}{\widetilde{\varsigma}}

\newcommand{\dbo}[1]{\overline{\overline{#1}}}



\title{The equivalence of pseudodifferential operators and their symbols via \v{C}ech-Dolbeault cohomology}

\date{}

\author{DAICHI KOMORI}

\begin{document}

\maketitle

\begin{abstract}
In this paper we construct the sheaf morphism from the sheaf of pseudodifferential operators to its symbol class.
Since the map is hard to construct directly, we realize it with two original ideas as follows.
First, to calculate cohomologies we use the theory of \v{C}ech-Dolbeault cohomology introduced by Honda, Izawa and Suwa in \cite{HIS}.
Secondly we construct a new symbol class, which is called the symbols of $C^\infty$-type.
These ideas enable us to construct the sheaf morphism, which is actually an isomorphism of sheaves.
\end{abstract}

\section*{Introduction}

The theory of hyperfunctions was introduced in \cite{SKK} and it gives us to conduct research of the systems of differential equations from a completely new perspective.
The essential idea of pseudodifferential operators was also given in \cite{SKK}, and later, Kashiwara and Kawai gave the explicit definition in \cite{KK}.\par

The class of pseudodifferential operators is sufficiently large class of differential operators and it contains truly important differential operators such as 
the differential operators of fractional order and those of infinite order.\par

Since the sheaf $\shE$ of pseudodifferential operators is explicitly defined by using the sheaf cohomology, for the study of $\shE$ in analytic category Kataoka in \cite{K} introduced symbols of pseudodifferential operators by the aid of the Radon transformations.
Moreover Aoki in \cite{A1},\cite{A2} established the symbol theory of $\shE$ and developed the study of the systems of differential equations of infinite order.
However, two fundamental problems are unresolved in their symbol theory:

\begin{enumerate}
\item The equivalence of the sheaf $\shE$ of pseudodifferential operators and its symbol class $\Q$ as sheaves.
\item The commutativity of the composition of pseudodifferential operators and the product of symbols through a symbol map.
\end{enumerate}

In his theory Aoki calculated the cohomological expression of a stalk of $\shE$ by using \v{C}ech cohomology.
In general we have to construct the \v{C}ech coverings which consist of Stein open sets.
For a global case, however, such coverings are hard to be found when we manipulate the cohomological expression of $\shE$.\par

In recent years the first problem has been solved by Aoki, Honda and Yamazaki in \cite{AHY}.
They introduce a new space with one apparent parameter and construct the sheaf morphism on it.
However, their construction is complicated and the more concise solution is desired as the foundation of symbol theory.\par

The aim of this paper is to construct the sheaf morphism from the sheaf $\shE$ of pseudodifferential operators to its symbol class by making use of \v{C}ech-Dolbeault cohomology.
Honda, Izawa and Suwa in \cite{HIS} find that the local cohomology groups with coefficients in the sheaf $\mathscr{O}$ of holomorphic functions is isomorphic to the cohomology group which is induced from double complex consisting of \v{C}ech coverings and the Dolbeault complex.
As the theory of \v{C}ech-Dolbeault cohomology is based on $C^\infty$ forms we can use convenient techniques such as a partition of unity, controlling the support by cutoff functions, and so on.
\par

As mentioned above, while we can apply useful techniques to $\shE$ via the \v{C}ech-Dolbeault cohomology , we have still some difficulties to construct the morphism from \v{C}ech-Dolbeault cohomology of $\shE$ to the symbol class $\Q$ since the symbol class $\Q$ is based on the theory of holomorphic functions.
To overcome this difficulty we introduce a new symbol class $\Qinf$, which consists of symbols of $C^\infty$ type.
Finally we can realize the morphism between $\shE$ and $\Q$ with concrete integration cycles.
\par

The plan of this paper is as follows.
The Section 1 provides some notations and definitions.
In Section 2 we introduce the \v{C}ech-Dolbeault cohomology of the sheaf $\shE$ of pseudodifferential operators.
Thanks to the study of Kashiwara and Schapira \cite{KS2} it is known that the section $\shE(V)$ on an open cone $V$ is represented by the inductive limit of local cohomology groups.
We apply the theory of \v{C}ech-Dolbeault cohomology to this cohomological expression.
In Section 3 we define new symbol class and prove that the new symbol class $\Qinf$ is isomorphic to the classical symbol class $\Q$ which is introduced by Aoki.
While the classical symbol theory is based on holomorphic functions, the \v{C}ech-Dolbeault cohomology is based on the $C^\infty$ functions.
Therefore it is hard to construct the map from \v{C}ech-Dolbeault cohomology to the classical symbol class.
We realize the map via new symbol class in the next section.
In Section 4 we construct the morphism $\varsigma$ from $\shE$ to $\Qinf$ by using the \v{C}ech-Dolbeault expression of $\shE$.
We also give the well-definedness of $\varsigma$.
In the appendix we prove the commutativity of the symbol map introduced by Aoki and the morphism constructed in Section 4.
For this purpose we introduced the \v{C}ech-Dolbeault cohomology with general coverings.
\par

The author would like to thank Professor Naofumi Honda in Hokkaido University for useful discussions and appropriate advices.

\section{Preliminaries}
Through this paper we shall follow the notations and definitions introduced below. \par
We denote by $\mathbb{Z},\mathbb{R}$ and $\mathbb{C}$ the sets of integers, of real numbers and of complex numbers, respectively.\par

Let $M$ be a real manifold of dimension $n$ and $X$ a complexification of $M$.
Set the diagonal set
\[
\Delta_X=\{(z,z')\in X\times X \mid z=z' \}.
\]
We write $\Delta$ instead of $\Delta_X$ if there is no risk of confusion.
One denotes by $p_1$ and $p_2$ the first and the second projections from $X\times X$ to $X$, respectively.

One denotes by $\tau:TX\rightarrow X$ the canonical projection from the tangent bundle to $X$ and $\pi:T^*X\rightarrow X$ that from the cotangent bundle to $X$.\par
Let $\omega$ be a $(p,q)$-form with coefficients in $C^\infty$-functions, and $\partial_z$ and $\bp_z$ the Dolbeault operators with respect to the variable $z$, that is, for a local coordinate $z=(z_1,z_2,\dots,z_n)$, the form $\omega$ can be written by
\[
\omega=\sum_{|I|=p,|J|=q}f_{IJ}(z)dz^I\wedge d\bar{z}^J.
\]
Moreover the Dolbeault operators are written by
\begin{align*}
\partial_z \omega = \sum^n_{i=1} \sum_{|I|=p,|J|=q} \frac{\partial}{\partial z_i}f_{IJ}(z)d z_i\wedge dz^I\wedge d\bar{z}^J, \\
\bp_z \omega = \sum^n_{i=1} \sum_{|I|=p,|J|=q} \frac{\partial}{\partial \bar{z}_i}f_{IJ}(z)d\bar{z}_i\wedge dz^I\wedge d\bar{z}^J.
\end{align*}

\begin{defi}
We define several sheaves:
\begin{enumerate}
\item Let $\mathscr{O}^{(p)}_X$ be the sheaf of holomorphic $p$-forms on $X$.
In particular $\mathscr{O}^{(0)}_X=\mathscr{O}_X$ is the sheaf of holomorphic functions on  $X$.
\item We denote by $or_X$ and $or_{M/X}=\mathscr{H}^{m}_M(\mathbb{Z}_X)$ the orientation sheaf on $X$ and the relative orientation sheaf on $M$, respectively.
\item Set $\Omega^{(n)}_X=\mathscr{O}^{(n)}_X\underset{\mathbb{C}_X}{\otimes} or_X$ and $\mathscr{O}^{(0,n)}_{X\times X} = \mathscr{O}_{X\times X}\underset{p^{-1}_2\mathscr{O}_X}{\otimes} p^{-1}_2\Omega^{(n)}_X$.
\item One denotes by $\Cinf{p}{q}$ the sheaf of $(p,q)$-forms with coefficients in $C^\infty$ on $X$.
\item One denotes by $\shE$ the sheaf of pseudodifferential operators on $T^*X$.
\end{enumerate}
\end{defi}
Let $(z;\zeta)$ be a local coordinate of $T^*X$.
Set $\ring{T}^*X=T^*X\setminus T^*_XX$ where $T^*_XX$ is the zero section.
We identify $T^*_\Delta(X\times X)$ with $T^*X$ by the map
\begin{equation}\label{ciso}
(z,z;\zeta,-\zeta)\mapsto (z;\zeta),
\end{equation}
which is induced from the first projection $p_1:X\times X \rightarrow X$.

\begin{defi}
Let $V$ be a set in $\ring{T}^*X$.
The set $V$ is called a cone, or equivalently called a conic set in $\ring{T}^*X$ if and only if
\[
(z;\zeta)\in V \Rightarrow (z;t\zeta)\in V \mbox{ for any }t\in\mathbb{R}_+.
\]
\end{defi}

\begin{rema}
Let $V$ be a set in $T^*X$.
We say that $V$ is convex (resp. conic, resp. proper) if for any $z \in \pi(V)$, the set $\pi^{-1}(z)\cap V$ is convex (resp. conic, resp. proper).
Recall that a cone is said to be proper if its closure contains no lines.
\end{rema}
Let $V$ and $V'$ be subsets in $T^*X$.
We write $V'\Subset V$ if $V'$ is a relatively compact set in $V$ for the usual topology.

\begin{defi}\label{rcc}
Let $V$ be an open cone in $\ring{T}^*X$.
A set $W\subset V$ is an infinitesimal wedge of type $V$ at infinity if for any $K\Subset V$ there exists $\delta>0$ such that
\[
K_\delta=\{(z;t\zeta) \mid (z;\zeta)\in K,\,t>\delta \} \subset W.
\]
In what follows $W$ is called the infinitesimal wedge of type $V$ for short.
\end{defi}

\begin{defi}\label{csub}
Let $V$ and $V'$ be cones in $\ring{T}^*X$ with $V'\subset V$.
The cone $V'$ is a relatively compact cone in $V$ if there exists a relatively compact set $K$ of $V$ such that
\[
V'=\{(z;t\zeta) \mid t\in\mathbb{R}_+,(z;\zeta)\in K\}.
\]
To clarify the differences, one denotes $V'\cSubset V$ if $V'$ is a relatively compact cone in $V$.
\end{defi}

\section{The sheaf $\shE$ of pseudodifferential operators and its \v{C}ech-Dolbeault expression}

First of all we briefly recall the sheaf $\shE$ of pseudodifferential operators.
Let $X$ be a complex manifold of dimension  $n$.
The sheaf $\shE$ of pseudodifferential operators on $T^*X$ is defined by
\begin{equation}
\mathscr{E}^\mathbb{R}_X=H^n(\mu_\Delta(\mathscr{O}^{(0,n)}_{X\times X})),
\end{equation}
where $\mu_\Delta(\mathscr{O}^{(0,n)}_{X\times X})$ is the microlocalization of $\mathscr{O}^{(0,n)}_{X\times X}$ along the diagonal set $\Delta$.
One denotes by $\shEx$ the stalk of $\shE$ at a point $z^*\in T^*X$.\par

Let $V$ be a subset of $T^*X$.
We denote by $V^\circ$ the polar set of $V$, that is, $V^\circ$ is defined by
\[
V^\circ=\{y\in TX \mid \tau(y) \in \pi(V)\ \mbox{and}\ {\rm Re}\,\langle x,y \rangle\geq 0\ \mbox{for all}\ x\in{\pi^{-1}}\circ\tau(y)\cap V\}.
\]

Then the following theorem is essential.

\begin{theo}[\cite{KS1}, Theorem 4.3.2]{\label{KS1}}
Let $V$ be an open convex cone in $T^*X$.
We have
\begin{equation}\label{cohomex}
\shE(V) = \ilim[U,G]H^n_{G\cap U}(U;\mathscr{O}^{(0,n)}_{X\times X}),
\end{equation}
where $U$ ranges through the family of open subsets of $X\times X$ such that $U\cap \Delta=\pi(V)$ and $G$ through the family of closed subsets of $X\times X$ such that $C_\Delta(G)\subset V^\circ$.
\end{theo}

Next we recall the \v{C}ech-Dolbeault cohomology introduced by Suwa \cite{HIS},\cite{Suwa}.
Let $M$ be a closed subset of $X$, $V_0=X\setminus M$ and $V_1$ an open neighborhood of $M$ in $X$.
For a covering $\mathcal{V}=\{V_0,V_1\}$ of $X$ we set
\begin{equation}
\Cinf{p}{q}(\mathcal{V}) = \Cinf{p}{q}(V_0)\oplus\Cinf{p}{q}(V_1)\oplus\Cinf{p}{q-1}(V_{01}),
\end{equation}
where $V_{01}=V_0\cap V_1$.
We also set the differential $\bar{\vartheta}:\Cinf{p}{q}\rightarrow \Cinf{p}{q+1}$ by
\begin{equation}
\bar{\vartheta}(\omega_0,\omega_1,\omega_{01})=(\bp \omega_0,\bp\omega_1,\omega_1-\omega_0-\bp\omega_{01}).
\end{equation}
Then $\bar{\vartheta}\circ\bar{\vartheta}=0$ is easily shown and the pair $(\Cinf{p}{\bullet}(\mathcal{V}),\bar{\vartheta})$ is a complex.

\begin{defi}
The \v{C}ech-Dolbeault cohomology $H^{p,q}_{\bar{\vartheta}}(\mathcal{V})$ of $\mathcal{V}$ of type $(p,q)$ is the $q$-th cohomology of the complex $(\Cinf{p}{\bullet}(\mathcal{V}),\bar{\vartheta})$.
\end{defi}

Next we consider the subcomplex of $(\Cinf{p}{\bullet}(\mathcal{V}),\bar{\vartheta})$ defined below.
Let $\mathcal{V}'=\{V_0\}$ be a covering of $X\setminus M$.
We set
\begin{equation*}
\Cinf{p}{q}(\mathcal{V},\mathcal{V}')=\{ (\omega_0,\omega_1,\omega_{01}) \in\Cinf{p}{q}(\mathcal{V}) \mid \omega_0=0\}=\Cinf{p}{q}(V_1)\oplus\Cinf{p}{q}(V_{01}).
\end{equation*}
Then the pair $(\Cinf{p}{\bullet}(\mathcal{V},\mathcal{V}'),\bar{\vartheta})$ is a subcomplex of $(\Cinf{p}{\bullet}(\mathcal{V}),\bar{\vartheta})$.

\begin{defi}\label{CDdef}
The \v{C}ech-Dolbeault cohomology $H^{p,q}_{\bar{\vartheta}}(\mathcal{V},\mathcal{V}')$ is the $q$-th cohomology of the complex $(\Cinf{p}{\bullet}(\mathcal{V},\mathcal{V}'),\bar{\vartheta})$.
\end{defi}

We have the following proposition.
\begin{prop}[\cite{HIS}, Proposition 4.6]\label{uniqueC}
The \v{C}ech-Dolbeault cohomology $H^{p,q}_{\bar{\vartheta}}(\mathcal{V},\mathcal{V}')$ is independent of the choice of $V_1$ and determined uniquely up to isomorphism.
\end{prop}
Therefore we can choose $X$ as $V_1$, and hereafter $H^{p,q}_{\bar{\vartheta}}(\mathcal{V},\mathcal{V}')$ is also denoted by $H^{p,q}_{\bar{\vartheta}}(X,X\setminus M)$.

\begin{theo}[\cite{HIS}, Theorem 4.9]\label{CDT}
There is a canonical isomorphism
\begin{equation}
H^{p,q}_{\bar{\vartheta}}(X,X\setminus S) \simeq H^q_S(X\,;\,\mathscr{O}^{(p)}_X).
\end{equation}
\end{theo}

Applying Theorem \ref{CDT} to the cohomology $H^n_{G\cap U}(U;\mathscr{O}^{(0,n)}_{X\times X})$ in Theorem \ref{KS1} we get the \v{C}ech-Dolbeault expression of $\shE$.

\begin{defi}
The sheaf $C^{\infty,(p,q;r)}_{X\times X}$ is the sheaf of $(p+q,r)$-forms with coefficients in $C^\infty$-functions which are holomorphic $p$-forms with respect to the first variable, holomorphic $q$-forms with respect to the second variable and antiholomorphic $r$-forms with respect to the first and the second variables.
In other words, for a local coordinate $(z_1,z_2)$ of $X\times X$ and for an open subset $V$ of $X\times X$, a form $f(z_1,z_2)\in C^{\infty,(p,q;r)}_{X\times X}(V)$ is written by
\[
f(z,z')=\sum_{|I|=p,|J|=q,|K|=r}f_{IJK}(z_1,z_2)dz^I_1\wedge dz^J_2\wedge d\bar{z}^K,
\]
where each $f_{IJK}(z_1,z_2)$ is a $C^\infty$-function on $V$.
\end{defi}

Set $V_0=U\setminus G$, $V_1 = U$ and $V_{01}=V_0\cap V_1 = U\setminus G$.
For coverings $\mathcal{V}=\{V_0,V_1\}$ of $U$ and $\mathcal{V}'=\{V_0\}$ of $U\setminus G$, we define
\[
C^{\infty,(p,q;r)}_{X\times X}(\mathcal{V},\mathcal{V}')=C^{\infty,(p,q;r)}_{X\times X}(V_1)\oplus C^{\infty,(p,q;r-1)}_{X\times X}(V_{01}).
\]
The differential $\bar{\vartheta}:C^{\infty,(p,q;r)}_{X\times X}(\mathcal{V},\mathcal{V}')\rightarrow C^{\infty,(p,q;r+1)}_{X\times X}(\mathcal{V},\mathcal{V}')$ is also given as usual, and the pair $(C^{\infty,(p,q;\bullet)}_{X\times X}(\mathcal{V},\mathcal{V}'),\bar{\vartheta})$ is a complex.
\begin{defi}
The $r$-th \v{C}ech-Dolbeault cohomology $H^{p,q,r}_{\bar{\vartheta}}(\mathcal{V},\mathcal{V}')$ is the $r$-th cohomology of the complex $(C^{\infty,(p,q;\bullet)}_{X\times X}(\mathcal{V},\mathcal{V}'),\bar{\vartheta})$.
\end{defi}

Thanks to Proposition \ref{uniqueC} and Theorem \ref{CDT} we have the following.
\begin{theo}\label{CD-classical}
There is a canonical isomorphism
\begin{equation}
H^{0,n,n}_{\bar{\vartheta}}(U,U\setminus G) \simeq H^n_{G\cap U}(U\,;\,\mathscr{O}^{(0,n)}_{X\times X}).
\end{equation}
\end{theo}
Thus the  section of $\shE$ on an open convex cone $V$ is expressed by
\[
\shE(V) = \ilim[U,G] H^n_{G\cap U}(U\,;\,\mathscr{O}^{(0,n)}_{X\times X}) = \ilim[U,G] H^{0,n,n}_{\bar{\vartheta}}(U,U\setminus G),
\]
where $U$ and $G$ run through the same sets as those in Theorem \ref{KS1}.

\section{Two symbol classes}
While the classical symbol theory $\mathfrak{S}/\mathfrak{N}$ of $\shE$ is based on holomorphic functions, the \v{C}ech-Dolbeault expression of $\shE$ is based on $C^\infty$-functions, and hence it is difficult to construct the map from \v{C}ech-Dolbeault expression to the classical symbol class directly.
In this section we construct a new symbol class which is of $C^\infty$-type and show that the new symbol class is isomorphic to the classical symbol class.

\subsection{The sheaf $\mathfrak{S}/\mathfrak{N}$ of classical symbols}
Let us review the classical symbol theory introduced by Aoki \cite{A1},\cite{A2}.
Let $z^*=(z;\zeta)$ be a local coordinate system of $T^*X$.
We construct two conic sheaves $\mathfrak{S}$ and $\mathfrak{N}$ on $\ring{T}^*X$.

\begin{defi}
Let $V$ be an open cone in $\ring{T}^*X$.
\begin{enumerate}
\item A function $f(z,\zeta)$ is called a symbol on $V$ if the following conditions hold.
\begin{enumerate}
\renewcommand{\labelenumii}{$(\roman{enumii})$}
\item There exists an infinitesimal wedge $W$ of type $V$ such that
\[
f(z,\zeta)\in\mathscr{O}_{T^*X}(W).
\]

\item For any open cone $V'\cSubset V$ there exists an infinitesimal wedge $W'\subset W$ of type $V'$ such that $f(z,\zeta)$ satisfies the following condition:\par
For any constant $h>0$, there exists a constant $C>0$ such that
\begin{equation}\label{infra}
|f(z,\zeta)| \leq C\cdot e^{h|\zeta|} \ \mbox{ on } W'.
\end{equation}
\end{enumerate}
\item A symbol $f(z,\zeta)$ on $V$ is called a null-symbol if for any open cone $V'\cSubset V$ there exist an infinitesimal wedge $W'\subset W$ of type $V'$ and constants $h>0$ and $C>0$ such that
\begin{equation}\label{small}
|f(z,\zeta)|\leq C\cdot e^{-h|\zeta|} \ \mbox{ on } W'.
\end{equation}

\item We denote by $\mathfrak{S}(V)$ and $\mathfrak{N}(V)$ the set of all the symbols on $V$ and the set of all the null-symbols on $V$, respectively.
Moreover we set
\begin{alignat*}{1}
\mathfrak{S}_{z^*}&= \ilim[V\ni z^*]\mathfrak{S}(V), \\
\mathfrak{N}_{z^*}&= \ilim[V\ni z^*]\mathfrak{N}(V),
\end{alignat*}
where $V$ runs through the family of open conic neighborhoods of $z^*\in\ring{T}^*X$.
\end{enumerate}
\end{defi}

We can naturally extend the sheaves $\mathfrak{S}$ and $\mathfrak{N}$ to the sheaves on $T^*X$.
Define the sheaves $\mathfrak{S}|_{T^*_XX}$ and $\mathfrak{N}|_{T^*_XX}$ on the zero section $T^*_XX=X$ as follows.
\begin{enumerate}
\item Let $U$ be an open set in $X$.
The section $\mathfrak{S}|_{T^*_XX}(U)$ is a family of $f(z,\zeta)\in\mathscr{O}_{T^*X}(\pi^{-1}(U))$ which satisfies the condition below:\par
For any compact set $K\Subset U$ and for any constant $h>0$ there exists a constant $C>0$ such that
\[
|f(z,\zeta)|\leq C\cdot e^{h|\zeta|}\ \ \mbox{on $\pi^{-1}(K)$}.
\]
\item Set $\mathfrak{N}|_{T^*_XX}=0$.
\end{enumerate}
Then the sheaves $\mathfrak{S}$ and $\mathfrak{N}$ become ones on $T^*X$.

\par
Next we construct the quotient sheaf $\mathfrak{S}/\mathfrak{N}$.
\begin{prop}
Let $V$ be an open cone in $\ring{T}^*X$.
The section $\mathfrak{N}(V)$ is an ideal of $\mathfrak{S}(V)$.
\end{prop}

\begin{proof}
Let $f(z,\zeta)\in\mathfrak{N}(V)$ and $g(z,\zeta)\in\mathfrak{S}(V)$.
Then there exists an infinitesimal wedge $W$ of type $V$ such that $f(z,\zeta)$ and $g(z,\zeta)$ are holomorphic on $W$.
By the definition of $\mathfrak{N}$ for any $V'\cSubset V$ we can find an infinitesimal wedge $W'$ of type $V'$ and the constants $h>0$ and $C>0$ such that
\[
|f(z,\zeta)|\leq C\cdot e^{-h|\zeta|}.
\]
Similarly for $V'$, $W'$ and $h>0$ which are the same ones as above, we can find a constant $C'>0$ such that
\[
|g(z,\zeta)|\leq C'\cdot e^{\frac{1}{2}h|\zeta|}.
\]
Hence we obtain
\[
|f(z,\zeta)\cdot g(z,\zeta)|\leq C\cdot e^{-h|\zeta|}\cdot C'\cdot e^{\frac{1}{2}h|\zeta|}\leq CC'\cdot e^{-\frac{1}{2}h|\zeta|}.
\]
\end{proof}
One denotes by $\widehat{\mathfrak{S}/\mathfrak{N}}$ the presheaf defined by the correspondence for an open cone $V$ in $\ring{T}^*X$
\[
V\mapsto \mathfrak{S}(V)/\mathfrak{N}(V).
\]
Let $\mathfrak{S}/\mathfrak{N}$ be the associated sheaf to $\widehat{\mathfrak{S}/\mathfrak{N}}$.
We have the following exact sequence of sheaves
\begin{equation}\label{SNseq}
0\longrightarrow \mathfrak{N}\longrightarrow \mathfrak{S}\overset{\kappa_1}{\longrightarrow} \mathfrak{S}/\mathfrak{N}\longrightarrow 0.
\end{equation}
Here $\kappa_1$ is the composition of the canonical morphisms $\mathfrak{S}\rightarrow \widehat{\mathfrak{S}/\mathfrak{N}}\rightarrow \mathfrak{S}/\mathfrak{N}$,
and \eqref{SNseq} induces the long exact sequence
\[
0\rightarrow \mathfrak{N}(V) \rightarrow \mathfrak{S}(V) \rightarrow \Q(V) \rightarrow H^1(V;\mathfrak{N}) \rightarrow \cdots.
\]
To treat $\Q(V)$ as it is a quotient group $\mathfrak{S}(V)/\mathfrak{N}(V)$, we claim $H^1(V;\mathfrak{N})=0$ for a suitable $V$.
\begin{theo}\label{vanish}
Assume $X$ to be a complex vector space and let $\widetilde{V}$ be a closed cone in $\ring{T}^*X$.
Moreover assume that $\widetilde{V}$ satisfies the conditions C1, C2 and C3.
\begin{enumerate}
\renewcommand{\labelenumi}{C\arabic{enumi}.}
\item A family of conic open neighborhoods of $\widetilde{V}$ has a cofinal family which consists of Stein open cones in $\ring{T}^*X$.
\item The projection $\pi(\widetilde{V})$ is a compact set in $X$.
\item There exists $\zeta_0\in\mathbb{C}^n\setminus \{0\}$ such that
\[
\widetilde{V}\subset \{ (z;\zeta)\in \ring{T}^*X \mid z\in \pi(\widetilde{V}) , {\rm Re }\,\langle \zeta,\zeta_0 \rangle > 0  \}.
\]
\end{enumerate}
Then $H^k(\widetilde{V};\mathfrak{N})=0$ holds for any $k>0$.
\end{theo}

The conditions C1, C2 and C3 are collectively called condition C.
Before starting the proof we construct a soft resolution of $\mathfrak{N}$.

\begin{exam}
We can construct a closed cone $\widetilde{V}$ satisfying the above three conditions as follows.
Let $N$ be a natural number and $f_1(z),f_2(z),\dots,f_N(z)$ holomorphic functions on $X$.
Set
$$
B= \bigcap^N_{i=1} \{ |f_i(z)| \leq 1\},
$$
and assume $B$ to be compact, and let $\Gamma$ be a closed proper convex cone.
Then $\widetilde{V}=B\times \Gamma$ satisfies the second and the third conditions in Theorem \ref{vanish}.
A cofinal family of $B\times\Gamma$ is given in the following way.
We can take a family $\{B_\varepsilon\}_{\varepsilon\in \mathbb{R}_+}$ of open neighborhoods of $B$ as follows
$$
B_\varepsilon = \bigcap_{1\leq i \leq N} \{ |f_i(z)| < 1+\varepsilon\}.
$$
Since $\Gamma$ is a closed proper convex cone we can take a cofinal family $\{\Gamma_\lambda\}_{\lambda\in \Lambda}$ which consists of open convex conic neighborhoods of $\Gamma$.
Then the family $\{B_\varepsilon\times \Gamma_\lambda\}_{(\varepsilon,\lambda)\in \mathbb{R}_+\times \Lambda}$ is what we want.
\end{exam}

\begin{defi}
One define the radial compactification $\mathbb{D}_{\mathbb{C}^n}$ of $\mathbb{C}^n$ by
\[
\mathbb{D}_{\mathbb{C}^n}=\mathbb{C}^n\sqcup S^{2n-1}\infty.
\]
\end{defi}
We show the fundamental system of neighborhoods.
If $z_0$ belongs to $\mathbb{C}^n$ a family of fundamental neighborhoods of $z_0$ consists of open sets
\begin{equation*}
B_\varepsilon(z_0)=\{  z\in \mathbb{C}^n \mid |z-z_0|<\varepsilon \}
\end{equation*}
for $\varepsilon>0$, otherwise that of $z_0\infty\in S^{2n-1}\infty$ consists of open sets
\[
G_r(\Gamma)=\left\{ z\in\mathbb{C}^n \  \middle| \  |z|>r,\frac{z}{|z|}\in\Gamma\right\}
\sqcup \Gamma,
\]
where $r>0$ and $\Gamma$ is an open neighborhood of $z_0\infty$ in $S^{2n-1}\infty$.
\begin{figure}[H]
\centering
\begin{tikzpicture}

\draw[gray, name path=c] (0,0) -- (3,2);
\draw[gray, name path=d] (0,0) -- (2,3);

\draw[dotted,name path=s](2,0) arc (0:90:2);

\draw[name intersections={of=s and c}];
\tikzmath{
coordinate \c;
\c= (intersection-1);
};

\draw(3,2) arc (33.69:56.31:3.6055);
\draw(\c) arc (33.69:56.31:2);
\draw(\c) -- (3,2);

\draw[name intersections={of=s and d}];
\tikzmath{
coordinate \d;
\d= (intersection-1);
};
\draw(\d) -- (2,3);

\draw[gray] (2,2) to [out=340, in=110] (3,1);
\draw[gray] (2.9,2.155) to [out=20, in=110] (3.5,2);

\draw (2.7495,2.5495) node[above]{$z_0\infty$};
\draw (3.4,1.8) node[right]{$\Gamma$};
\draw (3.5,0.8) node{$G_r(\Gamma)$};
\draw (2,0) node[right]{$r$};
\draw (0,0) node[left]{$O$};

\fill (2.5495,2.5495) circle [radius=2pt];
\end{tikzpicture}
\caption{$G_r(\Gamma)$}

\end{figure}

One denote by $\dbo{V}$ the closure of $V$ in $\mathbb{D}_{\mathbb{C}^n}$.

\begin{defi}
The radial compactification $\hat{T}^*X$ of $T^*X$ with respect to the fiber is
\[
\hat{T}^*X=\bigsqcup_{z\in X} \overline{\overline{T^*_zX}}.
\]
Here $\dbo{T^*_zX}\simeq \dbo{\mathbb{C}^n_\zeta} = \mathbb{C}^n_\zeta\sqcup S^{2n-1}\infty$.
\end{defi}
The topology of $\hat{T}^*X$ is induced from that of $\mathbb{D}_{\mathbb{C}^n}$.
Let $V$ be an open set in $\hat{T}^*X$.
We introduce some sheaves on $\hat{T}^*X$.
\begin{defi}
\begin{enumerate}
\item Let $\tilde{L}_{2,loc}$ be the sheaf of rapidly decreasing locally $L^2$-functions.
That is, for an open set $V \subset \hat{T}^*X$,
a function $f(z,\zeta)$ belongs to $\tilde{L}_{2,loc}(V)$ if for any compact set $W$ in $V$ there exists a constant $h>0$ such that
\[
f(z,\zeta)\cdot e^{h|\zeta|}\in L^2(W \cap T^*X).
\]

\item Let $\tilde{L}^{(p,q)}_{2,loc}$ be the sheaf of $(p,q)$-forms with coefficients in $\tilde{L}_{2,loc}$.

\item The sheaf $\tilde{\mathscr{L}}^{(p,q)}_{2,loc}$ is the subsheaf of $\tilde{L}^{(p,q)}_{2,loc}$ defined below:\par
A $(p,q)$-form $f\in\tilde{L}^{(p,q)}_{2,loc}(V)$ belongs to $\tilde{\mathscr{L}}^{(p,q)}_{2,loc}(V)$ if $\bp f(z,\zeta)\in\tilde{L}^{(p,q+1)}_{2,loc}(V)$.
\end{enumerate}
\end{defi}

\begin{lemm}\label{Nsoft}
The following resolution is soft:
\begin{equation}\label{Nres}
0\rightarrow \mathfrak{N} \overset{\bp}{\rightarrow} \tilde{\mathscr{L}}^{(0,0)}_{2,loc} \overset{\bp}{\rightarrow} \tilde{\mathscr{L}}^{(0,1)}_{2,loc} \overset{\bp}{\rightarrow} \cdots \overset{\bp}{\rightarrow} \tilde{\mathscr{L}}^{(0,2n)}_{2,loc} \rightarrow 0.
\end{equation}

\end{lemm}
In order to prove Lemma \ref{Nsoft}, it suffices to show the sequence
\begin{equation}\label{exactL}
0 \rightarrow \tilde{\mathscr{L}}^{(0,0)}_{2,loc}(\widetilde{V}) \rightarrow \tilde{\mathscr{L}}^{(0,1)}_{2,loc}(\widetilde{V}) \rightarrow \cdots \overset{\bp}{\rightarrow} \tilde{\mathscr{L}}^{(0,2n)}_{2,loc}(\widetilde{V}) \rightarrow 0
\end{equation}
is exact for any $\widetilde{V}$ which satisfies the condition C.
Actually we obtain Lemma \ref{Nsoft} by applying the inductive limit $\ilim[{\rm Int}\,(\widetilde{V}) \ni z^*]$ to \eqref{exactL}.\par

The following theorem is crucial in the proof of the exactness of \eqref{exactL}.
\begin{theo}[\cite{L1}, Theorem 4.4.2]\label{Hor}
Let $\Omega$ be a pseudoconvex open set in $\mathbb{C}^n$ and $\varphi$ any plurisubharmonic function in $\Omega$.
For every $g\in L^{(p,q)}_2(\Omega,\varphi)$ with $\bp g=0$ there is a solution $u\in L^{(p,q)}_{2,loc}(\Omega)$ of the equation $\bp u=g$ such that
\[
\int_\Omega |u|^2e^{-\varphi}(1+|z|^2)^{-2}d\lambda \leq \int_\Omega|g|^2e^{-\varphi}d\lambda.
\]
\end{theo}
\begin{rema}
In Theorem \ref{Hor} we adopt H\"{o}rmander's notation.
A form $g\in L^{(p,q)}_2(\Omega,\varphi)$ is a $(p,q)$-form on $\Omega$ with coefficients in square integrable functions with respect to the measure $e^{-\varphi}d\lambda$.
\end{rema}
Now we show the exactness of \eqref{exactL}.
Let $\widetilde{V}$ be a closed cone satisfying the condition C and set $\tilde{f}(z,\zeta)\in\tilde{\mathscr{L}}^{(0,q+1)}_{2,loc}(\widetilde{V})$.
Then there exist a Stein open cone $V_1$ with $\widetilde{V}\cSubset V_1$ and $f(z,\zeta)\in\tilde{\mathscr{L}}^{(0,q+1)}_{2,loc}(V_1)$ such that $\tilde{f}(z,\zeta)=f(z,\zeta)$ on $\widetilde{V}$.
Fix $\zeta_0\in \mathbb{C}^n\setminus\{0\}$ satisfying the condition C3.
Particularly we can assume $|\zeta_0|=1$ without loss of generality.
By the definition of $\tilde{\mathscr{L}}^{(0,q+1)}_{2,loc}$ there exist an infinitesimal wedge $W_	1$ of type $V_1$ and a constant $h>0$ such that
\[
f(z,\zeta)\cdot e^{h|\zeta|} \in L^{(0,q+1)}_2(W_1).
\]
Let $V_2$ be a Stein open cone with $\widetilde{V}\cSubset V_2\cSubset V_1$.
By Definition \ref{rcc} there exists a compact set $K$ of $V_1$ such that
\[
V_2=\{(z;t\zeta) \mid t\in\mathbb{R}_+,(z,\zeta)\in K\}.
\]
For such a compact set $K$ we can find a constant $\delta>0$ such that
\[
K_\delta=\{(z;t\zeta) \mid (z;\zeta)\in K,\,t>\delta \} \subset W.
\]

\begin{figure}[H]
\centering
\begin{tikzpicture}

\draw (0,0) -- (3,1);
\draw (0,0) -- (1,3);
\draw (3,1) arc (18.435:71.565:3.1623);

\draw[gray] (-1.5,1.5) -- (1.5,-1.5);
\draw[->] (0,0) -- (0.5,0.5);
\draw (0.5,0.5) node[right]{$\zeta_0$};

\draw (2.8,1.2) to [out=290, in=180] (3.5,0.8);
\draw (1.5,0.8) to [out=290, in=180] (2.5,0);

\draw (1,3) to [out=251.565, in=135] (1.5,1.5);
\draw (3,1) to [out=198.435, in=315] (1.5,1.5);

\draw(2.5,0) node[right]{$V_1$};
\draw(3.5,0.8) node[right]{$W_1$};
\draw (0,0) node[left]{$O$};

\end{tikzpicture}
\begin{tikzpicture}

\draw[dotted] (0,0) -- (3,1);
\draw[dotted] (0,0) -- (1,3);
\draw (3,1) arc (18.435:71.565:3.1623);

\draw (0,0) -- (1.754,2.631);
\draw (0,0) -- (2.631,1.754);

\draw[gray] (0.1,3.1) -- (3.1,0.1);

\draw[gray] (-1.5,1.5) -- (1.5,-1.5);
\draw[->] (0,0) -- (0.5,0.5);
\draw (0.5,0) node[right]{$\zeta_0$};

\draw[dotted] (1,3) to [out=251.565, in=135] (1.5,1.5);
\draw[dotted] (3,1) to [out=198.435, in=315] (1.5,1.5);

\draw (2,2) to [out=40, in=150] (3,2.1);
\draw (3,2.1) node[right]{$W_2$};

\draw (1,1) to [out=270, in=150] (1.5,0);
\draw (1.5,0) node[right]{$V_2$};

\draw (0,0) node[left]{$O$};
\draw (3.5,0.1) node[above]{$H_{\zeta_0}(\delta)$};

\end{tikzpicture}
\caption{$W_1$ and $W_2$}

\end{figure}

Hence for an open half space
\[
H_{\zeta_0}(\delta)=\{(z;\zeta)\in T^*X \mid z\in\pi(V_1),\ {\rm Re}\,\langle \zeta-\delta\zeta_0,\zeta_0 \rangle>0\},
\]
the set $W_2 = V_2 \cap H_{\zeta_0}(\delta)$ is convex, relatively compact set of $V_1$ and a Stein infinitesimal wedge of type $V_2$.
Thus we can fix a small constant $h'<h$ such that a function $f(z,\zeta)\cdot e^{h'\langle \zeta_0,\zeta\rangle}$ satisfies the condition C3.
Here we set a plurisubharmonic function $\varphi$ by
\[
\varphi=-\log(d(z,\zeta)),
\]
where $d(z,\zeta)={\rm dist}\,((z;\zeta),\partial W_2)$.
Set $F(z,\zeta) = f(z,\zeta)\cdot e^{h'\langle \zeta_0,\zeta\rangle}$.
By Theorem \ref{Hor} we can find $u(z,\zeta)\in L^{(0,q)}_{2,loc}(W_2)$ such that $\bp u=F$ and
\begin{equation}\label{ineq1}
\int_{W_2}(|\zeta|^2+1)^{-2}|u(z,\zeta)|^2\cdot e^{-\varphi} d\lambda \leq \int_{W_2} |F(z,\zeta)|^2\cdot e^{-\varphi} d\lambda< \infty.
\end{equation}
Therefore we have
\[
\bp g=e^{-h'\langle \zeta_0,\zeta\rangle}\cdot \bp u(z,\zeta) = e^{-h'\langle \zeta_0,\zeta\rangle}\cdot F(z,\zeta) = f(z,\zeta).
\]
In particular such a $g(z,\zeta)$ belongs to $\tilde{\mathscr{L}}^{(0,q)}_{2,loc}(\widetilde{V})$ because of \eqref{ineq1} and the exactness of \eqref{exactL} has been proved.\par
Finally we can calculate the global cohomology $H^k(\widetilde{V};\mathfrak{N})$ by using the resolution $\Gamma(\widetilde{V};\tilde{\mathscr{L}}^{(0,\bullet)}_{2,loc})$ and get the vanishing $H^k(\widetilde{V};\mathfrak{N})=0$ of the higher global cohomology for an arbitrary natural number $k$.

\begin{coro}
Let $\widetilde{V}$ be a closed cone satisfying the condition C.
Then an arbitrary element $f(z,\zeta) \in\Q(\widetilde{V})$ is represented by some symbol $f'(z,\zeta)\in\mathfrak{S}(\widetilde{V})$.
\end{coro}
The claim follows immediately from the exactness of the sequence
\[
0\rightarrow \mathfrak{N}(\widetilde{V}) \rightarrow \mathfrak{S}(\widetilde{V}) \rightarrow \Q(\widetilde{V}) \rightarrow 0.
\]

\subsection{The sheaf $\mathfrak{S}^\infty/\mathfrak{N}^\infty$ of symbols of $C^\infty$-type}

In this subsection we introduce new symbol class, which is called symbols of $C^\infty$-type.
Let $V$ be an open cone in $\ring{T}^*X$ and $z^*=(z;\zeta)$ a local coordinate of $T^*X$.
We construct conic sheaves $\mathfrak{S}^\infty$ and $\mathfrak{N}^\infty$ on $\ring{T}^*X$.
\begin{defi}
One defines the sheaf $C^\infty_z\mathscr{O}_\zeta$ as follows.
\[
\mbox{A function $f(z,\zeta)\in C^\infty_z\mathscr{O}_\zeta(V)$} \quad \Longleftrightarrow
\begin{array}{l}
 \mbox{A function }f(z,\zeta)\mbox{ is a $C^\infty$-function on $V$}\\
 \mbox{and a holomorphic on $V$ in the second variable.}
\end{array}
\]
\end{defi}

\begin{rema}
The sheaf $C^\infty_z\mathscr{O}_\zeta$ is invariant under the coordinate transformation of $X$.

\end{rema}

\begin{defi}\label{symbolinf}
\begin{enumerate}
\item A function $f(z,\zeta)$ is said to be a null-symbol of $C^\infty$-type on $V$ if it satisfies the following conditions.\par
\begin{enumerate}
\renewcommand{\labelenumii}{$N\arabic{enumii}$.}
\item There exists an infinitesimal wedge $W$ of type $V$ such that
\[
f(z,\zeta)\in {C^\infty_z\mathscr{O}_\zeta}(W).
\]
\item For any open cone $V'\cSubset V$ there exist an infinitesimal wedge $W'\subset W$ of type $V'$ such that the following condition holds: For any multi-indices $\alpha=(\alpha_1,\alpha_2,\dots,\alpha_n)\in \mathbb{Z}^n_{\geq 0}$ and $\beta=(\beta_1,\beta_2,\dots,\beta_n)\in \mathbb{Z}^n_{\geq 0}$, there exists constants $h>0$ and $C>0$ such that
\[
\left|\frac{\partial^\alpha}{\partial z^\alpha} \frac{\partial^\beta}{\partial \bar{z}^\beta} f(z,\zeta)\right|\leq C\cdot e^{-h|\zeta|}\ \mbox{ on } W'.
\]

\end{enumerate}
\item A function $f(z,\zeta)$ is said to be a symbol of $C^\infty$-type on $V$ if it satisfies the following conditions.\par
\begin{enumerate}
\renewcommand{\labelenumii}{$S\arabic{enumii}$.}
\item There exists an infinitesimal wedge $W$ of type $V$ such that
\[
f(z,\zeta)\in {C^\infty_z\mathscr{O}_\zeta}(W).
\]
\item For any open cone $V'\cSubset V$ there exists an infinitesimal wedge $W'\subset W$ of type $V'$ such that the following condition holds:
For any multi-indices $\alpha=(\alpha_1,\alpha_2,\dots,\alpha_n)\in \mathbb{Z}^n_{\geq 0}$ and $\beta=(\beta_1,\beta_2,\dots,\beta_n)\in \mathbb{Z}^n_{\geq 0}$, and for any $h>0$ there exists a constant $C>0$ such that
\[
\left|\frac{\partial^\alpha}{\partial z^\alpha} \frac{\partial^\beta}{\partial \bar{z}^\beta} f(z,\zeta)\right|\leq C\cdot e^{h|\zeta|}\ \mbox{ on } W'.
\]
\item The derivative $\displaystyle \frac{\partial}{\partial \bar{z}_i} f(z,\zeta)$ is a null-symbol on $V$ for any $i=1,2,\dots,n$.
\end{enumerate}

\item We denote by $\mathfrak{S}^\infty(V)$ and $\mathfrak{N}^\infty(V)$ the set of all the symbols of $C^\infty$-type on $V$ and the set of all the null-symbols of $C^\infty$-type on $V$, respectively.
Moreover we set
\begin{alignat*}{1}
\mathfrak{S}^\infty_{z^*}&= \ilim[V\ni z^*]\mathfrak{S}^\infty(V), \\
\mathfrak{N}^\infty_{z^*}&= \ilim[V\ni z^*]\mathfrak{N}^\infty(V),
\end{alignat*}
where $V$ runs through the family of open conic neighborhoods of $z^*$.
\end{enumerate}
\end{defi}
As we showed in Subsection 3.1, we can extend the sheaf $\mathfrak{S}^\infty$ and $\mathfrak{N}^\infty$ to the sheaves on $T^*X$.
Set $\mathfrak{S}^\infty|_{T^*_XX}=\mathfrak{S}|_{T^*_XX}$ and $\mathfrak{N}^\infty|_{T^*_XX}=0$.
Then the sheaves $\mathfrak{S}^\infty$ and $\mathfrak{N}^\infty$ are well-defined on $T^*X$.

\begin{prop}
Let $V$ be an open cone in $\ring{T}^*X$.
Then $\mathfrak{N}^\infty(V)$ is an ideal of $\mathfrak{S}^\infty(V)$.
\end{prop}
\begin{proof}
Let $f(z,\zeta)\in\mathfrak{N}^\infty(V)$ and $g(z,\zeta)\in\mathfrak{S}^\infty(V)$.
Then we can take an infinitesimal wedge $W$ of type $V$ such that $f(z,\zeta)$ and $g(z,\zeta)$ are in $C^\infty_z\mathscr{O}_\zeta(W)$.
Then we have
\[
\frac{\partial^\alpha}{\partial z^\alpha} \frac{\partial^\beta}{\partial \bar{z}^\beta} (f(z,\zeta)\cdot g(z,\zeta))=\sum_{\substack{{0 \leq \alpha' \leq \alpha} \\ {0 \leq \beta' \leq \beta}}} \frac{\partial^{\alpha'}}{\partial z^{\alpha'}} \frac{\partial^{\beta'}}{\partial \bar{z}^{\beta'}} f(z,\zeta) \cdot \frac{\partial^{\alpha-\alpha'}}{\partial z^{\alpha-\alpha'}} \frac{\partial^{\beta-\beta'}}{\partial \bar{z}^{\beta-\beta'}} g(z,\zeta).
\]
By the assumption for any $V'\cSubset V$ there exists an infinitesimal wedge $W'\subset W$ of type $V'$ such that $f(z,\zeta)$ satisfies the following condition:
For any $h_{\alpha'\beta'}>0$ there exists a positive constant $C_{\alpha'\beta'}$ such that
\[
\left|\frac{\partial^{\alpha'}}{\partial z^{\alpha'}} \frac{\partial^{\beta'}}{\partial \bar{z}^{\beta'}} f(z,\zeta)\right|\leq C_{\alpha'\beta'}\cdot e^{h_{\alpha'\beta'}|\zeta|}.
\]
Similarly for the same $V'\cSubset V$ and the same $W'\subset W$ there exist constants $h'_{\alpha'\beta'}>0$ and $C'_{\alpha'\beta'}>0$ such that
\[
\left|\frac{\partial^{\alpha-\alpha'}}{\partial z^{\alpha-\alpha'}} \frac{\partial^{\beta-\beta'}}{\partial \bar{z}^{\beta-\beta'}} g(z,\zeta)\right|\leq C'_{\alpha'\beta'}\cdot e^{-h'_{\alpha'\beta'}|\zeta|}.
\]
Thus by taking $\displaystyle h_{\alpha'\beta'}=\frac{1}{2}h'_{\alpha'\beta'}$ we obtain
\begin{alignat*}{1}
\left| \frac{\partial^\alpha}{\partial z^\alpha} \frac{\partial^\beta}{\partial \bar{z}^\beta} (f(z,\zeta)\cdot g(z,\zeta)) \right| &\leq \sum_{\substack{{0 \leq \alpha' \leq \alpha} \\ {0 \leq \beta' \leq \beta}}} \left| \frac{\partial^{\alpha'}}{\partial z^{\alpha'}} \frac{\partial^{\beta'}}{\partial \bar{z}^{\beta'}} f(z,\zeta)\right| \cdot \left| \frac{\partial^{\alpha-\alpha'}}{\partial z^{\alpha-\alpha'}} \frac{\partial^{\beta-\beta'}}{\partial \bar{z}^{\beta-\beta'}} g(z,\zeta) \right|\\
& \leq \sum_{\substack{{0 \leq \alpha' \leq \alpha} \\ {0 \leq \beta' \leq \beta}}} C_{\alpha'\beta'} C'_{\alpha'\beta'} \cdot e^{\frac{1}{2}h'_{\alpha'\beta'}|\zeta|}\cdot e^{-h'_{\alpha'\beta'}|\zeta|} \leq mCe^{-\frac{1}{2}h|\zeta|},
\end{alignat*}
where $\displaystyle h=\underset{\substack{{0 \leq \alpha' \leq \alpha} \\ {0 \leq \beta' \leq \beta}}}{\min} \{ h'_{\alpha'\beta'} \}$, $C=\underset{\substack{{0 \leq \alpha' \leq \alpha} \\ {0 \leq \beta' \leq \beta}}}{\max}\{C_{\alpha'\beta'} C'_{\alpha'\beta'}\}$ and $m = \#\{ (\alpha',\beta') \mid 0\leq \alpha' \leq \alpha, 0\leq \beta' \leq \beta \}$.
This completes the proof.
\end{proof}

One denotes by $\widehat{\Qinf}$ the presheaf defined by the correspondence
\[
V\mapsto \mathfrak{S}^\infty(V)/\mathfrak{N}^\infty(V),
\]
where $V$ is an open cone in $\ring{T}^*X$, and let $\mathfrak{S}^\infty/\mathfrak{N}^\infty$ be an associated sheaf to $\widehat{\Qinf}$.
We have the following exact sequence of sheaves.
\begin{equation}
0\longrightarrow \mathfrak{N}^\infty\longrightarrow \mathfrak{S}^\infty\overset{\kappa_2}{\longrightarrow} \mathfrak{S}^\infty/\mathfrak{N}^\infty\longrightarrow 0.
\end{equation}
Here $\kappa_2$ is the composition of the canonical morphisms $\mathfrak{S}^\infty\rightarrow \widehat{\mathfrak{S}^\infty/\mathfrak{N}^\infty}\rightarrow \Qinf$.
By the same argument in the previous subsection we want the exactness of the sequence on $\widetilde{V}$
\[
0\rightarrow \mathfrak{N}^\infty(\widetilde{V}) \rightarrow \mathfrak{S}^\infty(\widetilde{V}) \rightarrow \Qinf(\widetilde{V}) \rightarrow 0,
\]
where $\widetilde{V}$ satisfies the condition C, and this exactness is guaranteed by the following argument:\par
We have the commutative diagram
\[
\xymatrix{
0 \ar[r] & \mathfrak{N}(\widetilde{V}) \ar[r] \ar[d]^{\iota_2(\widetilde{V})} & \mathfrak{S}(\widetilde{V}) \ar[r]^(.4){\kappa_1(\widetilde{V})} \ar[d]^{\iota_1(\widetilde{V})} & \Q(\widetilde{V}) \ar[r] \ar[d]^{\iota(\widetilde{V})} & 0 \\
0 \ar[r] &\mathfrak{N}^\infty(\widetilde{V}) \ar[r] & \mathfrak{S}^\infty(\widetilde{V}) \ar[r]_(.4){\kappa_2(\widetilde{V})} & \Qinf(\widetilde{V}) &,
}
\]
where $\iota_1(\widetilde{V})$ and $\iota_2(\widetilde{V})$ are canonical inclusions and horizontal sequences are exact.
Assuming $\iota(\widetilde{V})$ to be an isomorphism,
the surjectivity of $\kappa_2(\widetilde{V})$ follows from the fact that the composition $\iota(\widetilde{V})\circ \kappa_1(\widetilde{V})$ is a surjective.\par

It shall be proved in the next subsection that $\iota$ is an isomorphism between $\Q$ and $\Qinf$.

\begin{coro}
Let $\widetilde{V}$ be a closed cone satisfying the condition C.
An arbitrary element $f(z,\zeta) \in\Qinf(\widetilde{V})$ is represented by some symbol $f'(z,\zeta)\in\mathfrak{S}^\infty(\widetilde{V})$.
\end{coro}

\subsection{The equivalence of two symbol classes}

In this subsection we may prove the equivalence of $\Q$ and $\Qinf$.\par
By the definitions of classical symbols and symbols of $C^\infty$-type, there exist canonical inclusions
\[
\iota_1:\mathfrak{S} \hookrightarrow \mathfrak{S}^\infty,\ \ \ \iota_2:\mathfrak{N} \hookrightarrow \mathfrak{N}^\infty,
\]
which induce the morphism
\[
\iota:\mathfrak{S}/\mathfrak{N}\longrightarrow\mathfrak{S}^\infty/\mathfrak{N}^\infty.
\]

\begin{theo}{\label{iso}}
The induced morphism
\[
\iota:\mathfrak{S}/\mathfrak{N} \longrightarrow \mathfrak{S}^\infty/\mathfrak{N}^\infty
\]
is an isomorphism of sheaves.
\end{theo}
Obviously we have $\Q|_{T^*_XX}=\Qinf|_{T^*_XX}=\mathscr{D}^\infty_X$.
Hence Theorem \ref{iso} holds on the zero section $T^*_XX$ and it is sufficient to prove Theorem \ref{iso} on $\ring{T}^*X$.\par
Since we have already obtained the map
\[
\iota:\Q \longrightarrow \Qinf,
\]
one shows $\iota_{z^*}$ to be an isomorphism of stalks at $z^*\in\ring{T}^*X$.\par
For this purpose we prepare the following proposition.
As the problem is local, we may assume that $T^*X\simeq \mathbb{C}^n_z\times \mathbb{C}^n_\zeta$ until the end of this subsection.
In addition we can take $z^*=z^*_0=(0\,;1,0,\dots,0)$ without loss of generality.
Let $D=D_1(r_1,0)\times D_2(r_2,0)\times \dots \times D_n(r_n,0)$ be a polydisc in $\mathbb{C}^n_z$ where $D_i(r_i,0)$ is an open disc in $\mathbb{C}$ whose radius is $r_i$ and the center is at the origin.
Set
\[
V=D\times \Gamma,
\]
where $\Gamma$ is an open convex cone containing $(1,0,\dots,0)\in\mathbb{C}^n_\zeta$. \par
We denote by $\mathfrak{N}^{\infty,(p,q)}$ the sheaf of $(p,q)$-forms on $T^*X$ with coefficients in $\mathfrak{N}^\infty$.

\begin{prop}\label{exist}
Let $V=D\times \Gamma$ be an open set defined above and let $f\in\mathfrak{N}^{\infty,(p,q)}(V)$ satisfy $\bp_z f=0$.
For any polydisc $D'\Subset D$ we can find $u\in\mathfrak{N}^{\infty,(p,q-1)}(V')$ with $V'=D'\times \Gamma$ such that $\bp_z u=f$ on $V'$.
\end{prop}
\begin{proof}
By the induction with respect to $k$, we prove that the lemma is true if $f$ does not contain $d\bar{z}_{k+1},\dots,d\bar{z}_n$.
If $k=0$, it is obvious that $f=0$.
Assuming that it has been proved when $k$ is replaced by $k-1$, we write
\begin{alignat*}{1}
f &= d\bar{z}_k\wedge g + h,\\
\displaystyle g &= \psum[|I|=p]\psum[|J|=q]g_{IJ}dz^I\wedge d\bar{z}^J.
\end{alignat*}
Here $g$ is a sum of $(p,q)$-forms on $V$ with coefficients in $C^\infty_z\mathscr{O}_\zeta$ and $h$ is a sum of $(p,q+1)$-forms on $V$ with coefficients in $C^\infty_z\mathscr{O}_\zeta$.
Moreover $g$ and $h$ do not contain $d\bar{z}_k,\dots,d\bar{z}_n$ and $\sum'$ means that we sum only over increasing multi-indices.
Since $\bar{\partial}f=0$ holds, we have
\[
\frac{\partial g_{IJ}}{\partial \bar{z}_j}=0
\]
for $j>k$ such that $g_{IJ}$ is analytic in these variables.\\
We want to construct the solution $G_{IJ}$ of the equation
\[
\frac{\partial G_{IJ}}{\partial \bar{z}_k}=g_{IJ}.
\]
For this purpose we fix a $C^\infty$-function $\psi$ on $D_k(r_k,0)$ with compact support such that $\psi(z_k)=1$ on a neighborhood $D''\subset D$ of $\overline{D'}$, and set
\begin{alignat*}{1}
G_{IJ} &= \frac{1}{(2\pi\sqrt{-1})} \iint \frac{1}{(\tau-z_k)} \psi(\tau)g_{IJ}(z_1,\dots,z_{k-1},\tau,z_{k+1},\dots,z_n)d\tau\wedge d\bar{\tau}\\
&= -\frac{1}{(2\pi\sqrt{-1})}\iint \frac{1}{\tau} \psi(z_k-\tau)g_{IJ}(z_1,\dots,z_{k-1},z_k-\tau,z_{k+1},\dots,z_n)d\tau\wedge d\bar{\tau}.
\end{alignat*}
The last integral representation shows that $G_{IJ}$ is a $C^\infty_z\mathscr{O}_\zeta$ function, and thus we just confirm that $G_{IJ}$ satisfies the condition $N2.$ in Definition \ref{symbolinf}.\par

Let $\alpha=(\alpha_1,\alpha_2,\dots,\alpha_n)\in \mathbb{Z}^n_{\geq 0}$ and $\beta=(\beta_1,\beta_2,\dots,\beta_n)\in \mathbb{Z}^n_{\geq 0}$ be multi-indices.
Hereafter $g_{IJ}(z_1,\dots,z_{k-1},z_k-\tau,z_{k+1},\dots,z_n)$ is also denoted by $g_{IJ}(z_k-\tau)$ for short.
Then we have
\begin{alignat*}{1}
\left| \frac{\partial^\alpha}{\partial z^\alpha} \frac{\partial^\beta}{\partial \bar{z}^\beta} G_{IJ} \right| &= \frac{1}{2\pi} \left| \iint \frac{1}{\tau} \cdot \frac{\partial^\alpha}{\partial z^\alpha} \frac{\partial^\beta}{\partial \bar{z}^\beta} (\psi(z_k-\tau)g_{IJ}(z_k-\tau)) d\tau\wedge d\bar{\tau} \right| \\
&= \frac{1}{2\pi} \left| \iint \frac{1}{\tau} \cdot \frac{\partial^\alpha}{\partial \tau^\alpha} \frac{\partial^\beta}{\partial \bar{\tau}^\beta} (\psi(z_k-\tau)g_{IJ}(z_k-\tau)) d\tau\wedge d\bar{\tau} \right|.
\end{alignat*}
We can calculate the integrand as follows.
\begin{alignat*}{1}
& \frac{1}{\tau} \frac{\partial^\alpha}{\partial \tau^\alpha} \frac{\partial^\beta}{\partial \bar{\tau}^\beta} (\psi(z_k-\tau)g_{IJ}(z_k-\tau)) \\
= & \sum_{0\leq \alpha' \leq \alpha} \sum_{0 \leq \beta' \leq \beta} \frac{1}{\tau} \frac{\partial^{\alpha'}}{\partial \tau^{\alpha'}} \frac{\partial^{\beta'}}{\partial \bar{\tau}^{\beta'}} \psi(z_k-\tau) \cdot \frac{\partial^{\alpha-\alpha'}}{\partial \tau^{\alpha-\alpha'}} \frac{\partial^{\beta-\beta'}}{\partial \bar{z}^{\beta-\beta'}} g_{IJ}(z_k-\tau).
\end{alignat*}
Since $\psi(z_k-\tau)$ has a compact support and $g_{IJ}$ is of $\mathfrak{N}^\infty$-type,
\[
\left| \iint \tau^{-1} \cdot \frac{\partial^{\alpha'}}{\partial \tau^{\alpha'}} \frac{\partial^{\beta'}}{\partial \bar{\tau}^{\beta'}} \psi(z_k-\tau) \cdot \frac{\partial^{\alpha-\alpha'}}{\partial \tau^{\alpha-\alpha'}} \frac{\partial^{\beta-\beta'}}{\partial \bar{z}^{\beta-\beta'}} g_{IJ}(z_k-\tau) d\tau\wedge\bar{\tau} \right| \leq C_{\alpha'\beta'}e^{-h_{\alpha'\beta'}|\zeta|}
\]
holds for some $C_{\alpha'\beta'}>0$ and $h_{\alpha'\beta'}>0$.
As the sets $\{\alpha\in\mathbb{Z}^n_{\geq 0} \mid 0\leq \alpha' \leq \alpha\}$ and $\{\beta\in\mathbb{Z}^n_{\geq 0} \mid 0 \leq \beta' \leq \beta\}$ are finite we obtain
\[
\left| \frac{\partial^\alpha}{\partial z^\alpha} \frac{\partial^\beta}{\partial \bar{z}^\beta} G_{IJ} \right| \leq Ce^{-h|\zeta|}
\]
for some $C>0$ and $h>0$.\par

Now we construct the solution $u$.
Set
\[
G=\psum[I,J]G_{IJ}dz^I\wedge d\bar{z}^J.
\]
It follows that
\[
\bar{\partial}G=\psum[I,J]\psum[j]\frac{\partial G}{\partial\bar{z}_j}d\bar{z}_j\wedge dz^I\wedge d\bar{z}^J=d\bar{z}_k\wedge g+h_1,
\]
where $h_1$ is the sum when $j$ runs from $1$ to $k-1$ and does not involve $d\bar{z}_k,\dots,d{\bar{z}}_n$.
Thus by the hypothesis of the induction we can find $v$ such that $\bar{\partial}v=f-\bar{\partial}G$ and $u=v+G$ satisfies the equation $\bar{\partial}u=f$.
The proof has been completed.
\end{proof}

Now let us prove Theorem \ref{iso}.
\begin{proof}
It is sufficient to show that the stalks of sheaves are isomorphic to each other.
One denotes by
\[
\iota_{z^*}:\mathfrak{S}_{z^*}/\mathfrak{N}_{z^*}\longrightarrow\mathfrak{S}^\infty_{z^*}/\mathfrak{N}^\infty_{z^*}
\]
the induced morphism from $\iota:\Q\rightarrow \Qinf$.
Since the injectivity of $\iota_{z^*}$ is obvious, we prove the surjectivity of it.\par
Set $F\in\mathfrak{S}^\infty_{z^*}$.
There exist a neighborhood $V=D\times \Gamma$ of $z^*$ and $f\in\mathfrak{S}(V)$ such that $f$ is a representative of $F$.
Then $f$ satisfies $\bp f\in\mathfrak{N}^\infty_{(0,1)}(V)$ and $\bp^2 f =0$.
By Proposition \ref{exist}  there exist $D'\Subset D$ and $g\in\mathfrak{N}^\infty(V')$ with $V'=D'\times \Gamma$ such that $\bp g=\bp f$ holds.
This implies $f-g\in\mathfrak{S}(V')$.
Set $F'=(f-g)_{z^*}$.
Obviously $\iota(F')_{z^*} = F$ holds and the surjectivity of $\iota_{z^*}$ has been proved.
\end{proof}

\section{The equivalence of $\shE$ and $\mathfrak{S}^\infty/\mathfrak{N}^\infty$}

In this section $X$ is assumed to be a complex vector space of dimension $n$.
We identify $X\times X$ with $TX$ by the map
\begin{equation}\label{ident}
\varrho:X\times X\ni (z,z')\mapsto (z,z'-z)\in TX,
\end{equation}
then we can easily see that the following diagram commutes
\[
\xymatrix{
X\times X \ar[rr]^\varrho \ar[rd]_{p_1} & & TX \ar[ld]^\tau \\
 & X. & 
}
\]
Here remark that $p_1$ is the first projection.
The aim of this section is to construct the sheaf morphism $\varsigma:\shE\longrightarrow\Qinf$ and prove the following theorem.
\begin{theo}\label{main}
The sheaf $\shE$ of pseudodifferential operators is isomorphic to the sheaf $\Q$ of classical symbols.
\end{theo}

\subsection{The map $\varsigma$ from $\shE$ to $\mathfrak{S}^\infty/\mathfrak{N}^\infty$}

Let $\widetilde{V}$ be a closed convex proper cone in $\ring{T}^*X$, and let $V$ and $V'$ be open convex proper cones with $\widetilde{V}\cSubset V'\cSubset V$.
Assume $\pi(\widetilde{V})$ is compact, and $\pi(V')$ and $\pi(V)$ are relatively compact sets.
Recall that we have the cohomological expression
\[
\shE(V)=\ilim[U,G]H^n_{G\cap U}(U;\mathscr{O}^{(0,n)}_{X\times X})
\]
under the suitable conditions.
If we have already obtained the map
\[
\widetilde{\varsigma}:H^n_{G\cap U}(U\,;\,\mathscr{O}^{(0,n)}_{X\times X})\rightarrow\Qinf(V'),
\]
by taking inductive limits $\ilim[U,G]$ and $\ilim[\widetilde{V}\cSubset V'\cSubset V]$ to $\widetilde{\varsigma}$ in this order we have
\[
\varsigma_{\widetilde{V}}:\shE(\widetilde{V})\longrightarrow \Qinf(\widetilde{V}).
\]
Hence our aim can be rephrased to construct the map $\widetilde{\varsigma}$ concretely.\par

Let $\gamma$ be a closed convex cone in $TX$.
We review the $\gamma$-topology on $TX$.
\begin{defi}
The $\gamma$-topology on $TX$ is the topology for which the open sets $\Omega$ satisfy:
\begin{enumerate}
\item $\Omega$ is open for the usual topology.
\item $\Omega\ring{+}\gamma =\Omega$.
\end{enumerate}
Here $\ring{+}$ is defined by
\[
\Omega\ring{+}\gamma=\bigsqcup_{z\in\tau(\Omega)}(\Omega_z+\gamma_z),
\]
where $\Omega_z=\Omega\cap \tau^{-1}(z)$ and $\gamma_z=\gamma\cap \tau^{-1}(z)$.
In particular if $\gamma_z=\emptyset$, set $\Omega_z+\gamma_z=\Omega_z$.
\end{defi}
An open set $V$ of $TX$ is called $\gamma$-open if $V$ is open in the sense of $\gamma$-topology.\par
Now let us construct the map
\[
\wsigma:H^n_{G\cap U}(U\,;\,\mathscr{O}^{(0,n)}_{X\times X})\longrightarrow \mathfrak{S}^\infty/\mathfrak{N}^\infty(V').
\]
Let $\Gamma_1$ and $\Gamma_2$ be open convex proper cones in $\ring{T}^*X$ such that
\begin{enumerate}
\item $V'\cSubset \Gamma_2 \cSubset \Gamma_1 \cSubset V$.
\item $G\cap U\subset \varrho^{-1}({\rm Int}\,(\Gamma^\circ_1))\cup \Delta$ in $p^{-1}_1(\pi(V'))$.
\end{enumerate}
Note that $\Delta$ is a diagonal set in $X\times X$.
\begin{rema}
By taking $U$ sufficiently small, we can guarantee the existence of $\Gamma_1$ since $G$ is tangent to $\varrho^{-1}(V^\circ)$ near the edge.
\end{rema}

\begin{figure}[H]
\centering
\vspace{10mm}
\begin{tikzpicture}

\draw[dotted] (0,0)--(2.2,4.3);
\draw[dotted] (0,0)--(4.3,2.2);
\draw[densely dotted] (0,0) -- (4.4,1.4);
\draw[densely dotted] (0,0) -- (1.4,4.4);
\draw (0,0) -- (4.5,0.5);
\draw (0,0) -- (0.5,4.5);
\draw (0,0) -- (4,3);
\draw (0,0) -- (3,4);

\draw[blue] (0,0) to [out=36.87, in=200] (3.3,1.9);
\draw[blue] (3.3,1.9) to [out=20, in=155] (4.5,1.7);

\draw[blue] (0,0) to [out=53.13, in= 250] (1.9,3.3);
\draw[blue] (1.9,3.3) to [out=70, in=300] (1.7,3.9);

\draw[dashed] (-1,3.8) to [out=0 ,in=90] (3.8,-0.2);

\draw (4.5,0.5) node[right]{$\varrho^{-1}({V'}^\circ)$};
\draw (4,3) node[right]{$\varrho^{-1}(V^\circ)$};
\draw (4.5,1.7) node[right]{$G$};
\draw (-1,3.8) node[below]{$U$};

\draw (2,4.8) node[right]{$\varrho^{-1}(\Gamma^\circ_1)$};
\draw (1.9,4.8) node[left]{$\varrho^{-1}(\Gamma^\circ_2)$};

\end{tikzpicture}
\caption{Geometrical relations in $X\times X$}
\end{figure}

We construct the paths of the integrations in the following way.
Set $\gamma_i=\Gamma^\circ_i$ for $i=1,2$.
Let $D_1$ and $D_2$ be open domains in $X\times X$ with $C^\infty$-smooth boundaries such that
\begin{enumerate}
\renewcommand{\labelenumi}{D$\arabic{enumi}$.}
\item $\varrho(D_i)$ is a $\gamma_i$-open set for $i=1,2$.
\item $\Delta_X(\pi(V'))\subset D_1$, where $\Delta_X:X\rightarrow X\times X$ is a diagonal embedding.
\item $\overline{D}_2 \cap p^{-1}_1(\pi(V')) \subset \varrho^{-1}({\rm Int}\,(\gamma_2))$.
\item $\overline{D}\cap p^{-1}_1(\pi(V'))\subset U$ for $D=D_1\setminus D_2$.
\item $\overline{E}\cap p^{-1}_1(\pi(V'))\subset U\setminus G$ for $E=\partial D_1\setminus D_2$.
\item $\partial D_1$ and $\partial D_2$ intersect transversally in $p^{-1}_1(\pi(V'))$
\item $p^{-1}_1(z)$ and $\partial D_1$ (resp. $\partial D_2$) intersect transversally for any $z\in\pi(V')$.
\end{enumerate}

\begin{exam}
Except the conditions D$6$ and D$7$, such $D_1$ and $D_2$ can be easily constructed as follows.
Set
\begin{alignat*}{1}
\widehat{D}_1&= \bigsqcup_{z\in \pi(V')} \{(z,v)\in T_zX \mid {\rm dist}_{T_zX}(v,\gamma_1\cap \tau^{-1}(z))<\varepsilon_1 \} , \\
\widehat{D}_2&= \bigsqcup_{z\in \pi(V')} \{(z,v)\in T_zX \mid {\rm dist}_{T_zX}(v,\tau^{-1}(z)\setminus\gamma_2))>\varepsilon_2 \},
\end{alignat*}
for $\varepsilon_1,\varepsilon_2>0$.
Then $D_1=\varrho^{-1}(\widehat{D}_1)$ and $D_2=\varrho^{-1}(\widehat{D}_2)$ satisfy the conditions D1$\sim$D5 if we take sufficiently small $\varepsilon_1$ and $\varepsilon_2$. \par
Moreover slightly modifying $D_1$ and $D_2$ we have $D_1$ and $D_2$ with $C^\infty$-boundaries which satisfy all the above conditions.
\end{exam}

\begin{figure}[H]
\centering
\vspace{5mm}
\begin{tikzpicture}

\draw[densely dotted] (0,0)--(2,4);
\draw[densely dotted] (0,0)--(4,2);

\draw (-0.2,0)--(2,4.4);
\draw (0,-0.2)--(4.4,2);

\draw (-0.2,0) to [out=243.435, in=206.565] (0,-0.2);

\draw (2.3,5) node[below]{$\varrho^{-1}(\widehat{D}_1)$};
\draw (3.5,2.5) node[below]{$\gamma_1$};
%\draw (1.5,4) node[above]{$-G_z$};

\draw[densely dotted] (0,0) -- (4,1);
\draw[densely dotted] (0,0) -- (1,4);
\draw (0.3,0.3) -- (4,1.225);
\draw (0.3,0.3) -- (1.25,4.1);

%\draw (3,1.3) to [out=195, in=250] (1.2,3);
\draw (4.8,0.9) node[above]{$\varrho^{-1}(\widehat{D}_2)$};
\draw (3.5,0.7) node[below]{$\gamma_2$};

\end{tikzpicture}
\caption{The figure of $\widehat{D}_1$ and $\widehat{D}_2$}
\end{figure}

Let $D_1$ and $D_2$ be open domains in $X\times X$ satisfying the above conditions.
Set $D=D_1\setminus D_2$, $E=\partial D_1\setminus D_2$, $D_z=D\cap p^{-1}_1(z)$ and $E_z=E\cap p^{-1}_1(z)$.
\begin{defi}
Let $u$ belong to $H^{0,n,n}_{\bar{\vartheta}}(U,U\setminus G)$ and let $\omega=(\omega_1,\omega_{01})$ be a representative of $u$.
One defines the map
\begin{equation}
\widetilde{\varsigma}:H^n_{G\cap U}(U\,;\,\mathscr{O}^{(0,n)}_{X\times X})=H^{0,n,n}_{\bar{\vartheta}}(U,U\setminus G) \rightarrow \mathfrak{S}^\infty/\mathfrak{N}^\infty(V')
\end{equation}
by
\begin{equation}
\wsigma(\omega)(z,\zeta)=\int_{D_z} \omega_1(z,z')\cdot e^{\langle z'-z,\zeta\rangle}-\int_{E_z}\omega_{01}(z,z')\cdot e^{\langle z'-z,\zeta\rangle}.
\end{equation}
\end{defi}

In the next paragraph we will show well-definedness of $\wsigma$.
More precisely we shall prove that $\wsigma$ is independent of the choices of the domains $D_1,D_2$ and a representative $\omega$ of $u$.

\subsection{Well-definedness of the map $\widetilde{\varsigma}$}

Let $V$ and $V'$ be the same as the sets given in the previous subsection.
\begin{prop}\label{welldef}
Let $\omega=(\omega_1,\omega_{01})$ be a representative of $u\in H^{0,n,n}_{\bar{\vartheta}}(U,U\setminus G)$.
The map $\wsigma$ has the following properties.
\begin{enumerate}
\item The image $\wsigma(\omega)$ belongs to $\mathfrak{S}^\infty(V')$.
\item The image $\wsigma(\omega)$ belongs to $\mathfrak{N}^\infty(V')$ if $\omega$ is equal to $0$ as an element of the \v{C}ech-Dolbeault cohomology.
\item The image $\wsigma(\omega)$ does not depend on the choices of $D_1$ and $D_2$.
\end{enumerate}
\end{prop}

For the proof of Proposition \ref{welldef} we expect $\wsigma$ and $\displaystyle \frac{\partial}{\partial z}$ (resp. $\wsigma$ and $\displaystyle \frac{\partial}{\partial \bar{z}}$) to be commutative.
However $\wsigma$ and $\displaystyle \frac{\partial}{\partial z}$ (resp. $\wsigma$ and $\displaystyle \frac{\partial}{\partial \bar{z}}$) do not commute in general since the paths $D_z$ and $E_z$ of the integrations $\wsigma$ depend on the variables $z$.
We begin by surmounting this difficulty.\par

For a fixed point $z_0\in X$ and a constant $\varepsilon>0$, set
\[
B(z_0,\varepsilon)=\{z\in X \mid |z-z_0|<\varepsilon\}.
\]
We give the subsets $\widetilde{D}(z_0,\varepsilon)$ and $\widetilde{E}(z_0,\varepsilon)$ in $X\times X$ by
\begin{alignat*}{1}
\widetilde{D}(z_0,\varepsilon) &= B(z_0,\varepsilon) \times p_2(D_{z_0})\subset X\times X, \\
\widetilde{E}(z_0,\varepsilon) &= B(z_0,\varepsilon) \times p_2(E_{z_0}) \subset X\times X.
\end{alignat*}
We also set
\[
\widehat{\varsigma}_{z_0}(\omega)=\int_{\widetilde{D}(z_0,\varepsilon)\cap p^{-1}_1(z_0)} \omega_1(z,z')\cdot e^{\langle z'-z,\zeta\rangle}-\int_{\widetilde{E}(z_0,\varepsilon) \cap p^{-1}_1(z_0)}\omega_{01}(z,z')\cdot e^{\langle z'-z,\zeta\rangle}.
\]

\begin{lemm}\label{prodlem}
Let $z_0\in \pi(V')$ and let $\omega=(\omega_1,\omega_{01})$ be a representative of $u\in H^{0,n,n}_{\bar{\vartheta}}(U,U\setminus G)$.
Then there exists a constant $\varepsilon>0$ such that the difference of the integrations
\[
\widehat{\varsigma}_{z_0}(\omega) - \wsigma(\omega)
\]
on $\pi^{-1}(B(z_0,\varepsilon))$ belongs to the null class $\mathfrak{N}^\infty(V'\cap \pi^{-1}(B(z_0,\varepsilon)))$.
\end{lemm}

\begin{proof}
It suffices to show that $\widehat{\varsigma}_{z_0}(\omega) - \wsigma(\omega)$ satisfies the condition $N2$ in Definition \ref{symbolinf}.
Since the domains $D_1$ and $D_2$ are smooth we can assume the following properties by taking sufficiently small $\varepsilon$.
\begin{enumerate}
\renewcommand{\labelenumi}{(\roman{enumi})}
\item $(z,z) \in \widetilde{D}(z_0,\varepsilon)$ for $z\in B(z_0,\varepsilon)$.
\item $\partial \widetilde{D}(z_0,\varepsilon) \setminus \widetilde{E}(z_0,\varepsilon) \subset \varrho^{-1}(\gamma_2)$.
\end{enumerate}
Here we take domains $D'$ and $\widetilde{D}'(z_0,\varepsilon)$ in $X \times X$ satisfying
\begin{enumerate}
\item $D'\subset D$ and $\widetilde{D}'(z_0,\varepsilon) \subset \widetilde{D}(z_0,\varepsilon)$.
\item $\partial D\setminus E\subset \partial D'$ and $\partial \widetilde{D}(z_0,\varepsilon)\setminus \widetilde{E}(z_0,\varepsilon) \subset \widetilde{D}'(z_0,\varepsilon)$.
\item For any $z\in B(z_0,\varepsilon)$, $D' \cap p^{-1}_1(z) = \widetilde{D}'(z_0,\varepsilon)\cap p^{-1}_1(z)$ on $U\setminus \varrho^{-1}(\gamma_2)$.
\end{enumerate}
Set
\begin{align*}
E'&= \partial D' \setminus (\partial D \setminus E),\\
\widetilde{E}'(z_0,\varepsilon)&=\partial \widetilde{D}'(z_0,\varepsilon)\setminus (\partial \widetilde{D}(z_0,\varepsilon) \setminus \widetilde{E}(z_0,\varepsilon)).
\end{align*}

By the Stoke's formula we have
\begin{align*}
\widehat{\varsigma}_{z_0}(\omega)&= \int_{\widetilde{D}(z_0,\varepsilon)\cap p^{-1}_1(z)} \omega_1(z,z')\cdot e^{\langle z'-z,\zeta\rangle}-\int_{\widetilde{E}(z_0,\varepsilon)\cap p^{-1}_1(z)}\omega_{01}(z,z')\cdot e^{\langle z'-z,\zeta\rangle}\\
&= \int_{\widetilde{D}'(z_0,\varepsilon)\cap p^{-1}_1(z)} \omega_1(z,z')\cdot e^{\langle z'-z,\zeta\rangle}-\int_{\widetilde{E}'(z_0,\varepsilon) \cap p^{-1}_1(z)}\omega_{01}(z,z')\cdot e^{\langle z'-z,\zeta\rangle},
\end{align*}
and
\begin{align*}
\wsigma(\omega) &= \int_{D_z} \omega_1(z,z')\cdot e^{\langle z'-z,\zeta\rangle}-\int_{E_z}\omega_{01}(z,z')\cdot e^{\langle z'-z,\zeta\rangle}\\
&= \int_{D'_z} \omega_1(z,z')\cdot e^{\langle z'-z,\zeta\rangle}-\int_{E'_z}\omega_{01}(z,z')\cdot e^{\langle z'-z,\zeta\rangle}.
\end{align*}
Here we write $D'_z$ and $E'_z$ instead of $D'\cap p^{-1}_1(z)$ and $E' \cap p^{-1}_1(z)$ for short, respectively.
Therefore we obtain
\begin{alignat*}{1}
& \widehat{\varsigma}_{z_0}(\omega)-\wsigma(\omega)\\
=& \left( \int_{\widetilde{D}(z_0,\varepsilon)\cap p^{-1}_1(z)} \omega_1(z,z')\cdot e^{\langle z'-z,\zeta\rangle}-
\int_{\widetilde{E}(z_0,\varepsilon)\cap p^{-1}_1(z)}\omega_{01}(z,z')\cdot e^{\langle z'-z,\zeta\rangle} \right) \\
& \hspace{60pt} -\left(\int_{D_z} \omega_1(z,z')\cdot e^{\langle z'-z,\zeta\rangle}-\int_{E_z}\omega_{01}(z,z')\cdot e^{\langle z'-z,\zeta\rangle}\right) \\
 = &\left( \int_{\widetilde{D}'(z_0,\varepsilon)\cap p^{-1}_1(z)} \omega_1(z,z')\cdot e^{\langle z'-z,\zeta\rangle}-\int_{D'_z} \omega_1(z,z')\cdot e^{\langle z'-z,\zeta\rangle} \right) \\
& \hspace{60pt} - \left( \int_{\widetilde{E}'(z_0,\varepsilon) \cap p^{-1}_1(z)}\omega_{01}(z,z')\cdot e^{\langle z'-z,\zeta\rangle} -\int_{E'_z}\omega_{01}(z,z')\cdot e^{\langle z'-z,\zeta\rangle} \right).
\end{alignat*}

As $\widetilde{D}'(z_0,\varepsilon)\cap p^{-1}_1(z_0)$ and $D'_z$ (resp. $\widetilde{E}'(z_0,\varepsilon)\cap p^{-1}_1(z_0)$ and $E'_z$) coincide in the domain $U\setminus \varrho^{-1}(\gamma_2)$ by the definitions of $\widetilde{D}'(z_0,\varepsilon)$ and $D'$ (resp. $\widetilde{E}'(z_0,\varepsilon)$ and $E'$), we get
\begin{equation}
\begin{aligned}\label{integralz}
&\int_{\widetilde{D}'(z_0,\varepsilon)\cap p^{-1}_1(z_0)} \omega_1(z,z')\cdot e^{\langle z'-z,\zeta\rangle}-\int_{D'_z} \omega_1(z,z')\cdot e^{\langle z'-z,\zeta\rangle} \\
= &\int_{\widetilde{D}'(z_0,\varepsilon)\cap p^{-1}_1(z_0) \cap \varrho^{-1}(\gamma_2)} \omega_1(z,z')\cdot e^{\langle z'-z,\zeta\rangle} - \int_{D'_z \cap \varrho^{-1}(\gamma_2)} \omega_1(z,z')\cdot e^{\langle z'-z,\zeta\rangle}.
\end{aligned}
\end{equation}
We estimate the former integration in \eqref{integralz}.
Let $\alpha\in \mathbb{Z}^n_{\geq 0}$ and $\beta \in \mathbb{Z}^n_{\geq 0}$ be multi-indices.
Since the path of the integration does not depend on $z$, there exist constants $C>0$ and $h>0$ such that
\begin{alignat*}{1}
& \left| \frac{\partial^\alpha}{\partial z^\alpha} \frac{\partial^\beta}{\partial \bar{z}^\beta} \int_{\widetilde{D}'(z_0,\varepsilon)\cap p^{-1}_1(z_0) \cap \varrho^{-1}(\gamma_2)} \omega_1(z,z')\cdot e^{\langle z'-z,\zeta\rangle} \right| \\
=& \left| \int_{\widetilde{D}'(z_0,\varepsilon)\cap p^{-1}_1(z_0) \cap \varrho^{-1}(\gamma_2)} \frac{\partial^\alpha}{\partial z^\alpha} \frac{\partial^\beta}{\partial \bar{z}^\beta} (\omega_1(z,z')\cdot e^{\langle z'-z,\zeta\rangle}) \right| \\
=& \left| \int_{\widetilde{D}'(z_0,\varepsilon)\cap p^{-1}_1(z_0) \cap \varrho^{-1}(\gamma_2)} \sum_{0\leq\alpha'\leq \alpha} \frac{\partial^{\alpha'}}{\partial z^{\alpha'}}\frac{\partial^\beta}{\partial z^\beta}\omega_1(z,z')\cdot \frac{\partial^{\alpha-\alpha'}}{\partial z^{\alpha-\alpha'}}e^{\langle z'-z,\zeta\rangle} \right| \\
\leq& \sum_{0\leq\alpha'\leq \alpha} \int_{\widetilde{D}'(z_0,\varepsilon)\cap p^{-1}_1(z_0) \cap \varrho^{-1}(\gamma_2)}  \left| \frac{\partial^{\alpha'}}{\partial z^{\alpha'}}\frac{\partial^\beta}{\partial z^\beta}\omega_1(z,z')\cdot \frac{\partial^{\alpha-\alpha'}}{\partial z^{\alpha-\alpha'}}e^{\langle z'-z,\zeta\rangle} \right| \\
=& \sum_{0\leq\alpha'\leq \alpha} \int_{\widetilde{D}'(z_0,\varepsilon)\cap p^{-1}_1(z_0) \cap \varrho^{-1}(\gamma_2)}  \left| \frac{\partial^{\alpha'}}{\partial z^{\alpha'}}\frac{\partial^\beta}{\partial z^\beta}\omega_1(z,z')\cdot (-\zeta)^{\alpha-\alpha'}e^{\langle z'-z,\zeta\rangle} \right| \leq Ce^{-h|\zeta|}.
\end{alignat*}

Next we estimate the latter integration in \eqref{integralz}.
Let $D'_z\cap\varrho^{-1}(\gamma_2)=\sqcup^N_{i=1} K_i$ be a partition such that each $K_i$ is bounded measurable subset in $p^{-1}_1(z)$.
Then we have
\begin{alignat*}{1}
\left| \frac{\partial^\alpha}{\partial z^\alpha} \frac{\partial^\beta}{\partial \bar{z}^\beta} \int_{D'_z \cap \varrho^{-1}(\gamma_2)} \omega_1(z,z')\cdot e^{\langle z'-z,\zeta\rangle} \right| =& \left| \frac{\partial^\alpha}{\partial z^\alpha} \frac{\partial^\beta}{\partial \bar{z}^\beta} \int_{\sqcup^N_{i=1} K_i} \omega_1(z,z')\cdot e^{\langle z'-z,\zeta\rangle} \right| \\
\leq& \sum^N_{i=1} \left| \frac{\partial^\alpha}{\partial z^\alpha} \frac{\partial^\beta}{\partial \bar{z}^\beta} \int_{K_i} \omega_1(z,z')\cdot e^{\langle z'-z,\zeta\rangle} \right|.
\end{alignat*}
Give the local coordinate $(z,z')=(z_1,\dots,z_n,z'_1,\dots,z'_n)$ for an open neighborhood $U_i$ of $K_i$ and consider the coordinate transformation
\[
\Phi_i: (z,z') \mapsto (z,\tilde{z}^i)
\]
such that $L_i=\Phi_i(K_i)$ is independent of the variables $z$.
Then we have
\[
\frac{\partial^\alpha}{\partial z^\alpha} \frac{\partial^\beta}{\partial \bar{z}^\beta} \int_{K_i} \omega_1(z,z')\cdot e^{\langle z'-z,\zeta\rangle}  = \frac{\partial^\alpha}{\partial z^\alpha} \frac{\partial^\beta}{\partial \bar{z}^\beta} \int_{L_i} \widetilde{\omega}_1(z,\tilde{z}^i)\cdot e^{\langle \Phi^{-1}(\tilde{z}^i)-z,\zeta \rangle} \cdot \left| J_{\Phi_i} \right|.
\]
Here $\widetilde{\omega}_1(z,\tilde{w}^i) = \omega_1(z,w) $ holds under the coordinate transform $\Phi_i$, and $J_{\Phi_i}$ is the Jacobian.

\begin{rema}
The existence of the coordinate transformation $\Phi_i$ is guaranteed by the condition $D7$.
\end{rema}

Since the domain $K_i$ is independent of the variables $z$ we obtain
\begin{alignat*}{1}
& \left| \frac{\partial^\alpha}{\partial z^\alpha} \frac{\partial^\beta}{\partial \bar{z}^\beta} \int_{K_i} \omega_1(z,z')\cdot e^{\langle z'-z,\zeta\rangle} \right| \\
=& \left| \frac{\partial^\alpha}{\partial z^\alpha} \frac{\partial^\beta}{\partial \bar{z}^\beta} \int_{L_i} \widetilde{\omega}_1(z,\tilde{z}^i)\cdot e^{\langle \Phi^{-1}(\tilde{z}^i)-z,\zeta \rangle} \cdot \left| J_{\Phi_i} \right| \right| \\
 \leq & \int_{L_i} \left| \frac{\partial^\alpha}{\partial z^\alpha} \frac{\partial^\beta}{\partial \bar{z}^\beta} (\widetilde{\omega}_1(z,\tilde{z}^i) \cdot \left| J_{\Phi_i} \right| \cdot e^{\langle \Phi^{-1}(\tilde{z}^i)-z,\zeta \rangle}) \right| \\
 \leq & \int_{L_i} \sum_{0\leq \alpha' \leq \alpha} \left| \frac{\partial^{\alpha'}}{\partial z^{\alpha'}} \frac{\partial^\beta}{\partial \bar{z}^\beta} (\widetilde{\omega}_1(z,\tilde{z}^i) \cdot \left| J_{\Phi_i} \right|) \cdot \frac{\partial^{\alpha-\alpha'}}{\partial z^{\alpha-\alpha'}} e^{\langle \Phi^{-1}(\tilde{z}^i)-z,\zeta \rangle} \right| \\
 \leq & \sum_{0\leq \alpha' \leq \alpha} \int_{L_i} \left| \frac{\partial^{\alpha'}}{\partial z^{\alpha'}} \frac{\partial^\beta}{\partial \bar{z}^\beta} (\widetilde{\omega}_1(z,\tilde{z}^i) \cdot \left| J_{\Phi_i} \right|) \cdot (-\zeta)^{\alpha-\alpha'} e^{\langle \Phi^{-1}(\tilde{z}^i)-z,\zeta \rangle} \right|.
\end{alignat*}
Moreover the absolute value of its derivative is bounded on $L_i$ since $\widetilde{\omega}_1(z,\tilde{z}^i) \cdot \left| J_{\Phi_i} \right|$ is a $C^\infty$-function.
For each $i$ there exists $M_i,C_i>0$ and $h_i>0$ such that
\begin{alignat*}{1}
  & \sum_{0\leq \alpha' \leq \alpha} \int_{L_i} \left| \frac{\partial^{\alpha'}}{\partial z^{\alpha'}} \frac{\partial^\beta}{\partial \bar{z}^\beta} (\widetilde{\omega}_1(z,\tilde{w}^i) \cdot \left| J_{\Phi_i} \right|) \cdot (-\zeta)^{\alpha-\alpha'} e^{\langle \Phi^{-1}(\tilde{z}^i)-z,\zeta \rangle} \right| \\
\leq & \sum_{0\leq \alpha' \leq \alpha} \int_{L_i} C_i\cdot \left|(-\zeta)^{\alpha-\alpha'} e^{\langle \Phi^{-1}(\tilde{z}^i)-z,\zeta \rangle} \right| \leq  M_i C_i \cdot e^{-h|\zeta|} \\
\end{alignat*}

Finally for $M=\max \{M_i\}, C=\max \{ C_i \}$ and $h=\min \{ h_i \}$ we have
\[
\sum^N_{i=1} \left| \frac{\partial^\alpha}{\partial z^\alpha} \frac{\partial^\beta}{\partial \bar{z}^\beta} \int_{K_i} \omega_1(z,z')\cdot e^{\langle z'-z,\zeta\rangle} \right|
\leq \sum^N_{i=1} M_i C_i e^{-h_i|\zeta|} \leq NMCe^{-h|\zeta|}.
\]
We can apply the same argument to
\[
\int_{\widetilde{E}'(z_0,\varepsilon) \cap p^{-1}_1(z)}\omega_{01}(z,z')\cdot e^{\langle z'-z,\zeta\rangle} -\int_{E'_z}\omega_{01}(z,z')\cdot e^{\langle z'-z,\zeta\rangle}
\]
and these complete the proof.
\end{proof}

Now we start the proof of Proposition \ref{welldef}.
In the following proof the Dolbeault operator $\bar{\partial}_z+\bar{\partial}_{z'}$ is denoted by $\bar{\partial}$ without notice.
\begin{proof}
\begin{enumerate}
\item We can divide $D_z$ into two subsets $D_{z,1}(\varepsilon)$ and $D_{z,2}(\varepsilon)$ in $p^{-1}_1(z)$ for sufficiently small $\varepsilon>0$ which have piecewise smooth boundaries such that
\begin{enumerate}
\item $G\cap D_{z,1}(\varepsilon)=\emptyset$.
\item $\partial E_z(\varepsilon)=\partial (\partial D_z\setminus E_z)$.
\item $D_{z,2}(\varepsilon)\subset (\varrho^{-1}(\gamma_2) \cup B((z,z),\varepsilon))\cap p^{-1}_1(z)$ where $B((z,z),\varepsilon)$ is an open ball in $X\times X$ with radius $\varepsilon$ whose center is at $(z,z)$.
\end{enumerate}
Here we set $E_z(\varepsilon)=\partial D_{z,1}(\varepsilon) \cap \partial D_{z,2}(\varepsilon)$.

\begin{figure}[H]
\centering
\begin{tikzpicture}

\fill[blue!30] (3.6,1.3) -- (0.2,-0.3) to [out=200, in=250] (-0.3,0.2) -- (1.3,3.6) -- (0.1,0.25) to [out=240, in=0] (0,0.2) to [out=180, in=135] (-0.1,-0.1) to [out=315, in=270] (0.2,0) to [out=90, in=30] (0.25,0.1) -- (3.6,1.3);

\draw[densely dotted] (0,0)--(2.2,4);
\draw[densely dotted] (0,0)--(4,2.2);

\draw[densely dotted] (0,0) -- (4,1);
\draw[densely dotted] (0,0) -- (1,4);

\draw[densely dotted] (0,0) to [out=30, in=190] (4,1.5);
\draw[densely dotted] (0,0) to [out=60, in=260] (1.5,4);

\draw[densely dotted] (3.6,1.3) -- (1.55,0.9) to [out=180, in=265] (0.9,1.55) -- (1.3,3.6);

\draw (3.6,1.3) -- (0.2,-0.3) to [out=200, in=250] (-0.3,0.2) -- (1.3,3.6) -- (0.1,0.25) to [out=240, in=0] (0,0.2) to [out=180, in=135] (-0.1,-0.1) to [out=315, in=270] (0.2,0) to [out=90, in=30] (0.25,0.1) -- (3.6,1.3);

\draw (3,4.2) node[left]{$\varrho^{-1}(\gamma_1)$};
\draw (1.2,4.2) node[left]{$\varrho^{-1}(\gamma_2)$};

\draw (4,1.5) node[right]{$G$};

\end{tikzpicture}
\vspace{5mm}
\begin{tikzpicture}

\fill[blue!30] (3.6,1.3) -- (1.55,0.9) to [out=190, in=260] (0.9,1.55) -- (1.3,3.6) -- (0.1,0.25) to [out=240, in=0] (0,0.2) to [out=180, in=135] (-0.1,-0.1) to [out=315, in=270] (0.2,0) to [out=90, in=30] (0.25,0.1) -- (3.6,1.3);

\draw (3.6,1.3) -- (1.55,0.9) to [out=190, in=260] (0.9,1.55) -- (1.3,3.6) -- (0.1,0.25) to [out=240, in=0] (0,0.2) to [out=180, in=135] (-0.1,-0.1) to [out=315, in=270] (0.2,0) to [out=90, in=210] (0.25,0.1) -- (3.6,1.3);

\draw[densely dotted] (3.6,1.3) -- (0.2,-0.3) to [out=200, in=250] (-0.3,0.2) -- (1.3,3.6);

\draw[densely dotted] (0,0)--(2.2,4);
\draw[densely dotted] (0,0)--(4,2.2);

\draw[densely dotted] (0,0) -- (4,1);
\draw[densely dotted] (0,0) -- (1,4);

\draw[densely dotted] (0,0) to [out=30, in=190] (4,1.5);
\draw[densely dotted] (0,0) to [out=60, in=260] (1.5,4);

\draw (3,4.2) node[left]{$\varrho^{-1}(\gamma_1)$};
\draw (1.2,4.2) node[left]{$\varrho^{-1}(\gamma_2)$};

\draw (4,1.5) node[right]{$G$};

\end{tikzpicture}
\caption{The subsets $D_{1,\varepsilon}$ and $D_{2,\varepsilon}$}
\end{figure}

\begin{figure}[H]
\centering
\begin{tikzpicture}

\draw[gray] (0,0) -- (4,1);
\draw[gray] (0,0) -- (1,4);

\draw[dotted] (0,0) circle (3);
\draw (3,-2.2) node[below]{radius $\varepsilon$};

\draw[gray] (0,0) to [out=25, in=190] (4,1.5);
\draw[gray] (0,0) to [out=65, in=260] (1.5,4);

\draw (4,1.2) -- (1.9,0.65) to [out=195, in=90] (1.5,0) to [out=270 ,in=315] (-1,-1) to [out=135 ,in=180] (0,1.5) to [out=0, in=255] (0.65,1.9) -- (1.2,4);

\draw (0.9,3.5) node[left]{$\varrho^{-1}(\gamma_2)$};
\draw (3.6,1.7) node[right]{$G$};
\draw (-1,-1) node[left]{$E(\varepsilon)$};
\end{tikzpicture}
\caption{$E_\varepsilon$ near the vertex}
\end{figure}

By the Stoke's formula we have
\begin{alignat*}{1}
\int_{D_{z,1}(\varepsilon)}\omega_1(z,z')\cdot e^{\langle z'-z,\zeta\rangle}
=& \int_{D_{z,1}(\varepsilon)}\bar{\partial}\omega_{01}(z,z')\cdot e^{\langle z'-z,\zeta\rangle} \\
=& \int_{E_z+E(\varepsilon)}\omega_{01}(z,z')\cdot e^{\langle z'-z,\zeta\rangle}.
\end{alignat*}
For any $h>0$, by retaking $\varepsilon$ small enough to satisfy ${\rm Re}\,\langle z'-z , \zeta \rangle\leq h|\zeta|$ in $B((z,z),\varepsilon)\cap p^{-1}_1(z)$, we can see that
\begin{align*}
&\left|\frac{\partial^\alpha}{\partial z^\alpha}\frac{\partial^\beta}{\partial \bar{z}^\beta}\wsigma(\omega)\right| \\
=&\left|\int_{D_{z,2}(\varepsilon)} \frac{\partial^\alpha}{\partial z^\alpha}\frac{\partial^\beta}{\partial \bar{z}^\beta}\left( \omega_1(z,z')\cdot e^{\langle z'-z,\zeta\rangle} \right) +\int_{E(\varepsilon)}\frac{\partial^\alpha}{\partial z^\alpha}\frac{\partial^\beta}{\partial \bar{z}^\beta} \left( \omega_{01}(z,z')\cdot e^{\langle z'-z,\zeta \rangle} \right) \right|\\
=&\left|\int_{D_{z,2}(\varepsilon)} \frac{\partial^\alpha}{\partial z^\alpha} \left( \frac{\partial^\beta}{\partial \bar{z}^\beta} \omega_1(z,z')\cdot e^{\langle z'-z,\zeta\rangle} \right) +\int_{E(\varepsilon)}\frac{\partial^\alpha}{\partial z^\alpha} \left( \frac{\partial^\beta}{\partial \bar{z}^\beta} \omega_{01}(z,z')\cdot e^{\langle z'-z,\zeta \rangle} \right) \right|\\
\leq&\left|\int_{D_{z,2}(\varepsilon)} \sum_{0\leq \alpha' \leq \alpha} \frac{\partial^{\alpha'}}{\partial z^{\alpha'}} \frac{\partial^\beta}{\partial \bar{z}^\beta} \omega_1(z,z')\cdot \frac{\partial^{\alpha-\alpha'}}{\partial z^{\alpha-\alpha'}} e^{\langle z'-z,\zeta\rangle} \right| \\
& \hspace{4cm}+ \left| \int_{E(\varepsilon)} \sum_{0\leq \alpha' \leq \alpha} \frac{\partial^{\alpha'}}{\partial z^{\alpha'}} \frac{\partial^\beta}{\partial \bar{z}^\beta} \omega_{01}(z,z')\cdot \frac{\partial^{\alpha-\alpha'}}{\partial z^{\alpha-\alpha'}} e^{\langle z'-z,\zeta \rangle} \right|\\
\leq& \sum_{0\leq \alpha' \leq \alpha} \int_{D_{z,2}(\varepsilon)} \left|\frac{\partial^{\alpha'}}{\partial z^{\alpha'}}\frac{\partial^\beta}{\partial \bar{z}^\beta}\omega_1(z,z')\right|\cdot \left|P_{\alpha'}(\zeta)\right| \cdot e^{{\rm Re}\,\langle z'-z,\zeta\rangle}\\
& \hspace{4cm} + \sum_{0\leq \alpha' \leq \alpha} \int_{E(\varepsilon)} \left|\frac{\partial^{\alpha'}}{\partial z^{\alpha'}}\frac{\partial^\beta}{\partial \bar{z}^\beta}\omega_{01}(z,z') \right| \cdot \left|P_{\alpha'}(\zeta)\right| \cdot e^{{\rm Re}\langle z'-z,\zeta\rangle}\\
 \leq & C\cdot e^{2h|\zeta|},
\end{align*}
where $P_{\alpha'}(\zeta)$ is some polynomial with respect to $\zeta$ and hence we obtain $|P_{\alpha'}(\zeta)|\leq e^{h|\xi|}$ for any $\alpha'$ and sufficiently small $\varepsilon$.

Next we check that $\displaystyle \frac{\partial}{\partial \bar{z}_i} \widetilde{\varsigma}(\omega)$ belongs to $\mathfrak{N}^\infty(V')$ for any $i=1,2,\dots,n$.
By the Stoke's formula and the facts that $(\bar{\partial}_z+\bar{\partial}_{z'})\omega_{01}=\omega_1$ and $\bar{\partial}_z\omega_1=-\bar{\partial}_{z'}\omega_1$, we obtain
\begin{alignat*}{1}
\left|\bar{\partial}_z\wsigma(\omega)\right| =&\left|\int_{D_{z,2}(\varepsilon)} \bar{\partial}_z\omega_1(z,z')\cdot e^{\langle z'-z,\zeta\rangle}+\int_{E(\varepsilon)}\bar{\partial}_z\omega_{01}(z,z')\cdot e^{\langle z'-z,\zeta \rangle}\right|\\
=&\left|-\int_{D_{z,2}(\varepsilon)} \bar{\partial}_{z'}\omega_1(z,z')\cdot e^{\langle z'-z,\zeta\rangle}\right.\\
& \hspace{20mm}\left. +\int_{E(\varepsilon)}(\omega_1(z,z')-\bp_{z'} \omega_{01}(z,z'))\cdot e^{\langle z'-z,\zeta \rangle}\right|\\
=&\left|\int_{\partial D_{z,2}(\varepsilon)} \omega_1(z,z')\cdot e^{\langle z'-z,\zeta\rangle}\right.\\
&\hspace{20mm}\left.-\int_{E(\varepsilon)}\omega_1(z,z')\cdot e^{\langle z'-z,\zeta \rangle}+\int_{\partial E(\varepsilon)}\omega_{01}(z,z')\cdot e^{\langle z'-z,\zeta \rangle}\right|\\
=&\left|\int_{\partial D_{z,2}(\varepsilon)\setminus E(\varepsilon)} \omega_1(z,z')\cdot e^{\langle z'-z,\zeta\rangle}+\int_{\partial E(\varepsilon)}\omega_{01}(z,z')\cdot e^{\langle z'-z,\zeta \rangle}\right|.
\end{alignat*}
By the construction of the partitions $D_{z,1}$ and $D_{z,2}$ we obtain $\partial{D_{z,2}(\varepsilon)}\setminus E(\varepsilon)\subset {\rm Int}\,(\varrho^{-1}(\gamma_2))$ and $\partial E(\varepsilon)\subset {\rm Int}\,(\varrho^{-1}(\gamma_2))$, and hence there exist positive constants $h$ and $C$ such that
\begin{alignat*}{1}
&\left|\int_{\partial D_{z,2}(\varepsilon)\setminus E(\varepsilon)} \omega_1(z,z')\cdot e^{\langle z'-z,\zeta\rangle}+\int_{\partial E(\varepsilon)}\omega_{01}(z,z')\cdot e^{\langle z'-z,\zeta \rangle}\right|\\
\leq& \int_{\partial{D_{z,2}(\varepsilon)}\setminus E(\varepsilon)} |\omega_1(z,z')|\cdot e^{-h|\zeta|}
+ \int_{\partial E(\varepsilon)} |\omega_{01}(z,z')|\cdot e^{-h|\zeta|}\\
\leq& C\cdot e^{-h|\zeta|},
\end{alignat*}
and we obtain the claim.

\item In addition to the assumption in the above proof we assume that $\omega$ is equal to $0$ as an element in the relative \v{C}ech-Dolbeault cohomology.
Then there exists $\tau=(\tau_1,\tau_{01})\in C^{\infty,(0,n,n-1)}_{X\times X}(\mathcal{V},\mathcal{V}')$ with $\bar{\vartheta}\tau=\omega$.
By substituting $(\omega_1,\omega_{01})$ with $(\bar{\partial}\tau_1,\tau_1-\bar{\partial}\tau_{01})$ we have
\begin{alignat*}{1}
\widetilde{\varsigma}(\omega) =&
\int_{D_{z,2}(\varepsilon)} \omega_1(z,z')\cdot e^{\langle z'-z,\zeta\rangle}
+ \int_{E(\varepsilon)}\omega_{01}(z,z')\cdot e^{\langle z'-z,\zeta \rangle}\\
= &\int_{D_{z,2}(\varepsilon)} \bar{\partial}\tau_1(z,z')\cdot e^{\langle z'-z,\zeta\rangle}
+ \int_{E(\varepsilon)}(\tau_1(z,z')
- \bar{\partial}\tau_{01}(z,z'))\cdot e^{\langle z'-z,\zeta \rangle}.
\end{alignat*}
Since the integrations
\[
\int_{D_{z,2}(\varepsilon)} \bar{\partial}_z\tau_1(z,z')\cdot e^{\langle z'-z,\zeta\rangle}, \quad \int_{E(\varepsilon)} \bar{\partial}_z\tau_{01}(z,z')\cdot e^{\langle z'-z,\zeta \rangle}
\]
vanish, we have
\begin{alignat*}{1}
&\int_{D_{z,2}(\varepsilon)} \bar{\partial}\tau_1(z',w)\cdot e^{\langle z'-z,\zeta\rangle}
+ \int_{E(\varepsilon)}(\tau_1(z,z')
- \bar{\partial}\tau_{01}(z,z'))\cdot e^{\langle z'-z,\zeta \rangle}\\
=& \int_{\partial{D_{z,2}(\varepsilon)}\setminus E(\varepsilon)} \tau_1(z,z')\cdot e^{\langle z'-z,\zeta\rangle}
+ \int_{\partial E(\varepsilon)} \tau_{01}(z,z')\cdot e^{\langle z'-z,\zeta \rangle}.
\end{alignat*}
By the same argument as in the proof of $1$, we can find constants $h>0$ and $C>0$ such that
\begin{alignat*}{1}
&\left|\int_{\partial{D_{z,2}(\varepsilon)}\setminus E(\varepsilon)} \tau_1(z,z')\cdot e^{\langle z'-z,\zeta\rangle}
+ \int_{\partial E(\varepsilon)} \tau_{01}(z,z')\cdot e^{\langle z'-z,\zeta \rangle}\right| \\
\leq& \int_{\partial{D_{z,2}(\varepsilon)}\setminus E(\varepsilon)} |\tau_1(z,z')|\cdot e^{-h|\zeta|}
+ \int_{\partial E(\varepsilon)} |\tau_{01}(z,z')|\cdot e^{-h|\zeta|} \leq C\cdot e^{-h|\zeta|}.
\end{alignat*}

\item The claim follows immediately from the argument in the above proofs.

\end{enumerate}
\end{proof}

The corollary below follows from Proposition \ref{welldef}.

\begin{coro}
The map $\wsigma$ is well-defined.
\end{coro}

In the next subsection we prove the main theorem.

\subsection{The proof of the main Theorem \ref{main}}

Now we show the proof of Theorem \ref{main}.
As a consequence of Subsections 4.1 and 4.2,
there exists a morphism
\[
\varsigma_{\widetilde{V}}: \shE(\widetilde{V}) \longrightarrow \Qinf(\widetilde{V})
\]
for any closed convex proper cone $\widetilde{V}$ in $\ring{T}^*X$ with $\pi(\widetilde{V})$ being compact.\par

Let $\widetilde{V}'$ be a closed convex proper cone contained in $\widetilde{V}$.
Then it follows from Proposition \ref{welldef} that the diagram below commutes:

\[
\xymatrix@R=30pt@C=20pt{
\shE(\widetilde{V}) \ar[r]^(.4){\varsigma_{\widetilde{V}}} \ar[d]& \Qinf(\widetilde{V}) \ar[d] \\
\shE(\widetilde{V}') \ar[r]^(.4){\varsigma_{\widetilde{V}'}} & \Qinf(\widetilde{V}').
}
\]
Since the family of closed convex proper cones in $\ring{T}^*X$ is a basis of sets on which a conic sheaf can be defined, the family $\{ \varsigma_{\widetilde{V}} \}_{\widetilde{V}}$ of morphisms gives a sheaf morphism on $\ring{T}^*X$
\[
\varsigma:\shE \longrightarrow \Qinf.
\]
The rest of the problem is whether the map $\varsigma$ is an isomorphism or not.
In particular it suffices to show the morphism $\varsigma_{z^*}:\mathscr{E}^\mathbb{R}_{X,z^*} \rightarrow (\Qinf)_{z^*}$ of stalks is isomorphic.
Assume that the following diagram commutes for an each point $z^*$.
(We may show it in the appendix since the proof is a little complicated.)
\begin{equation}\label{com2map}
\vcenter{
\xymatrix@R=20pt@C=50pt{
  & (\Q)_{z^*} \ar[dd]^{\iota_{z^*}} \\
\mathscr{E}^\mathbb{R}_{X,z^*} \ar[ru]^\sigma \ar[rd]_{\varsigma_{z^*}} & \\
  & (\Qinf)_{z^*}
}
}
\end{equation}

Here $\sigma$ is the symbol mapping defined by Aoki \cite{A2}.
The following theorem is essential.
\begin{theo}[\cite{A2}, Theorem 4.3 and Theorem 4.5]\label{sigmaiso}
The symbol mapping $\sigma$
\[
\sigma: \mathscr{E}^\mathbb{R}_{X,z^*} \to (\Q)_{z^*}
\]
 is an isomorphism of stalks.
\end{theo}

The vertical arrow in the diagram is an isomorphism by Theorem \ref{iso} and $\sigma$ is also isomorphic by the above theorem.
Therefore $\varsigma_{z^*}$ is also an isomorphism, which completes the proof of Theorem \ref{main}.

\appendix

\section{The compatibility of two symbol maps}

In the appendix we prove the commutativity of \eqref{com2map}.
Since the argument is local, we have the following identity map.
\[
T^*X \simeq X \times \mathbb{C}^n \ni (z;\zeta) = (z_1,\dots,z_n; \zeta_1,\dots,\zeta_n).
\]
Moreover assume that $X\subset \mathbb{C}^n$ is an open set with coordinate system $z=(z_1,\dots,z_n)$.

Fix the point $z^*$ so that we can consider the stalk of $\shE$ on it.
Then we regard $z^* = (0;\lambda,0\dots,0)$ for some $\lambda\in \mathbb{C}\setminus \{0\}$ without loss of generality. 
Recall that the stalk of the sheaf $\shE$ is expressed as the inductive limit of the cohomologies with suitable subsets $U$ and $G$.
\[
\mathscr{E}^\mathbb{R}_{X,z^*} = \ilim[U,G]H^n_G(U\,;\,\mathscr{O}^{(0,n)}_{X\times X}).
\]
On the other hand, in the theory of the \v{C}ech-Dolbeault representation of the stalk of $\shE$ is given in Definition \ref{CDdef} and Theorem \ref{CD-classical}.
\par
To express two cohomologies in the same class we give another \v{C}ech-Dolbeault expression.
In the second step we calculate the integration of two cohomology classes and finish the proof.

\subsection{Another \v{C}ech-Dolbeault expression of $\mathscr{E}^\mathbb{R}_X$}

We construct the \v{C}ech-Dolbeault cohomology of $\shE$ in which two cohomologies can be embedded.
(See \cite{Suwa}.)\par
Let $M$ be a closed set of codimension $m$ in $X$.
Moreover let $\mathcal{W}=\{W_i\}_{i\in I}$ and $\mathcal{W}'=\{W_i\}_{i\in I'}$ be open coverings of $X$ and $X\setminus M$, respectively.
Here $I$ and $I'$ are index sets with $I'\subset I$.
For a sheaf $\mathscr{S}$ we set
\[
C^q (\mathcal{W};\mathscr{S}) = \prod_{(\alpha_0,\alpha_1,\dots,\alpha_q)\in I^{q+1}} \mathscr{S}(W_{\alpha_0\alpha_1\dots\alpha_q}),
\]
where $W_{\alpha_0\alpha_1\dots\alpha_q} = W_{\alpha_0}\cap W_{\alpha_1}\cap \dots \cap W_{\alpha_q}$.
Note that $\sigma_{\alpha_0\dots\alpha_q} \in \mathscr{S}(W_{\alpha_0\alpha_1\dots\alpha_q})$ has the orientation, that is, the section $\sigma$ satisfies the formula
\[
\sigma_{\alpha_0\dots\alpha_i\alpha_{i+1}\dots\alpha_q} = -\sigma_{\alpha_0\dots\alpha_{i+1}\alpha_i\dots\alpha_q}.
\]
This implies that we have $\sigma = 0$ if $\alpha_i = \alpha_j$ for $i\neq j$.
The complex $(C^\bullet(\mathcal{W};\mathscr{S}),\delta)$ is called the \v{C}ech complex with coefficients in $\mathscr{S}$.
The coboundary operator $\delta$ is defined by
\[
(\delta \sigma)_{\alpha_0\alpha_1\dots\alpha_{q+1}} = \sum^{q+1}_{k=0} (-1)^k \sigma_{\alpha_0\dots\widehat{\alpha_k}\dots\alpha_{q+1}}.
\]
Here we set $\sigma_{\alpha_0\dots\widehat{\alpha_k}\dots\alpha_{q+1}} = \sigma_{\alpha_0\dots \alpha_{k-1} \alpha_{k+1}\dots\alpha_{q+1}}$.
We also define the relative \v{C}ech complex $(C^\bullet(\mathcal{W},\mathcal{W}';\mathscr{S}),\delta)$ as follows.
\[
C^q(\mathcal{W},\mathcal{W}';\mathscr{S}) = \{ \sigma\in C^q(\mathcal{W};\mathscr{S})\mid \sigma_{\alpha_0\dots\alpha_q} = 0 \text{ if } \alpha_0,\alpha_1,\dots,\alpha_q\in I' \}.
\]
Hereafter we denotes by $\mathscr{F}^q=C^{\infty,(0,n,q)}_{X\times X}$ until the end of this appendix.
Recall that we have the fine resolution of $\mathscr{O}^{(0,n)}_{X\times X}$:
\[
0 \longrightarrow \mathscr{O}^{(0,n)}_{X\times X} \longrightarrow \mathscr{F}^0 \overset{\bar{\partial}}{\longrightarrow} \mathscr{F}^1 \overset{\bar{\partial}}{\longrightarrow} \cdots \overset{\bar{\partial}}{\longrightarrow} \mathscr{F}^{2n} \longrightarrow 0.
\]

This induces the following double complex:
\[
\xymatrix@R=25pt@C=35pt{
 & \vdots \ar[d]^\delta & \vdots \ar[d]^\delta & \\
\cdots \ar[r]^(.3){(-1)^{q_1}\bar{\partial}} & C^{q_1}(\mathcal{W},\mathcal{W}';\mathscr{F}^{q_2}) \ar[r]^{(-1)^{q_1}\bar{\partial}} \ar[d]^\delta & C^{q_1}(\mathcal{W},\mathcal{W}';\mathscr{F}^{q_2+1}) \ar[d]^\delta \ar[r]^(.7){(-1)^{q_1}\bar{\partial}} & \cdots \\
\cdots \ar[r]^(.3){(-1)^{q_1+1}\bar{\partial}} & C^{q_1+1}(\mathcal{W},\mathcal{W}';\mathscr{F}^{q_2}) \ar[d]^(.4)\delta \ar[r]^{(-1)^{q_1+1}\bar{\partial}} & C^{q_1+1}(\mathcal{W},\mathcal{W}';\mathscr{F}^{q_2+1}) \ar[d]^(.4)\delta \ar[r]^(.7){(-1)^{q_1+1}\bar{\partial}} & \cdots \\
 & \vdots & \vdots & .
}
\]
We consider the associated single complex $(\mathscr{F}^\bullet(\mathcal{W},\mathcal{W}'),D)$ as follows.
\[
\mathscr{F}^q(\mathcal{W},\mathcal{W}') =\bigoplus_{q_1+q_2=q} C^{q_1}(\mathcal{W},\mathcal{W}';\mathscr{F}^{q_2}),\quad D=\delta+(-1)^{q_1}\bar{\partial}.
\]
\begin{defi}
The \v{C}ech-Dolbeault cohomology $H^q(\mathcal{W},\mathcal{W}';\mathscr{F}^\bullet)$ of $(\mathcal{W},\mathcal{W}')$ with coefficients in $\mathscr{F}^\bullet$ is the $q$-th cohomology of $(\mathscr{F}^\bullet(\mathcal{W},\mathcal{W}'),D)$.
\end{defi}

We describe the differential $D$ a few more details.
A cochain $\xi\in \mathscr{F}^q(\mathcal{W},\mathcal{W}')$ may be expressed as $\xi = (\xi^{(q_1)})_{0\leq q_1\leq q}$ with $\xi^{(q_1)} \in C^{q_1}(\mathcal{W},\mathcal{W}';\mathscr{F}^{q-q_1})$.
Then $D:\mathscr{F}^q(\mathcal{W},\mathcal{W}')\rightarrow\mathscr{F}^{q+1}(\mathcal{W},\mathcal{W}')$ is given by
\[
(D\xi)^{(q_1)} = \delta \xi^{(q_1-1)} + (-1)^{q_1} \bar{\partial} \xi^{q_1},\quad 0\leq q_1\leq q+1,
\]
where we set $\xi^{(-1)}=\xi^{(q+1)}=0$.

\begin{rema}
Assume $\mathcal{W}$ consisting of two open sets $X$ and $X\setminus M$, and $\mathcal{W}'$ consisting of only one open set $X\setminus M$.
Then the \v{C}ech-Dolbeault cohomology $H^q(\mathcal{W},\mathcal{W}';\mathscr{F}^\bullet)$ corresponds to the \v{C}ech-Dolbeault cohomology defined in Section 2.
\end{rema}

\subsection{Two cohomological expressions in the \v{C}ech-Dolbeault cohomology}

The goal of this subsection is to express two cohomologies in the same cohomology class.
For more details of \v{C}ech representation, see \cite{A2} and \cite{SKK}.
\par

By the discussion so far we can get three cohomological expressions $H^n (\mathscr{O}^{(0,n)}_{X\times X}; \mathcal{W},\mathcal{W}')$, $H^{0,n,n}_{\bar{\vartheta}}(\mathcal{V},\mathcal{V}')$ and $H^{0,n,n}_{\bar{\vartheta}}(\mathcal{W},\mathcal{W}')$ of $\mathscr{E}^\mathbb{R}_{X,z^*}$.
Set $U_r$ and $G_{r,\varepsilon}$ as follows.
\begin{alignat*}{1}
U_r = & \{ (z,z') \in X\times X \mid |z|<r, |z'-z|<r \}, \\
G_{r,\varepsilon} = & \{ (z,z') \in U_r \mid |z'_1-z_1|\geq \varepsilon |z'_i-z_i| \ (2 \leq i \leq n), \\
 & \hspace{20mm}  - {\rm Re}\,(\lambda (z'_1-z_1))\geq \varepsilon|{\rm Im}\,(\lambda (z'_1-z_1))|\  \}.
\end{alignat*}
Thanks to the study of Kashiwara-Kawai \cite{KK} we have
\begin{equation}\label{cohomologyEX}
\shEx = \ilim[\substack{r \rightarrow 0 \\ \varepsilon \rightarrow 0}]H^n_{G_{r,\varepsilon}}(U_r;\mathscr{O}^{(0,n)}_{X\times X}).
\end{equation}
We write $U$ and $G$ instead of $U_r$ and $G_{r,\varepsilon}$ respectively if there is no risk of confusion.
We also set several open sets as follows.
\begin{alignat*}{1}
V_0 &= W_0 = U,\quad V_1 = U\setminus G,\quad V_{01} = V_0\cap V_1 = U\setminus G, \\
W_1 &= \{ (z,z') \in U \mid -{\rm Re}\,(\lambda (z'_1-z_1)) < \varepsilon|{\rm Im}\, (\lambda (z'_1-z_1))| \}, \\
W_i &= \{ (z,z') \in U \mid |z'_1-z_1| < \varepsilon |z'_i-z_i| \}\quad (2\leq i \leq n).
\end{alignat*}
We can easily see that the followings are open coverings of $U$:
\[
\mathcal{V} = \{V_0,V_1\},\quad
\mathcal{W} = \{W_0,W_1,\dots,W_n\},
\]
and the followings are open coverings of $U\setminus G$:
\[
\mathcal{V}' = \{V_1\},\quad \mathcal{W}' = \{W_1,W_2,\dots,W_n\}.
\]

Fix small $r$ and $\varepsilon$.
We compute the cohomology $H^n_{G}(U;\mathscr{O}^{(0,n)}_{X\times X})$ by applying the \v{C}ech cohomology with respect to the \v{C}ech coverings $( \mathcal{W},\mathcal{W}' )$.
Since $\mathcal{W}$ and $\mathcal{W}'$ are Stein coverings, by Leray's theorem we have the exact sequence
\[
\bigoplus^n_{i=1}\Gamma(W_{\hat{i}} ; \mathscr{O}^{(0,n)}_{X\times X}) \longrightarrow \Gamma (W;\mathscr{O}^{(0,n)}_{X\times X}) \longrightarrow H^n_{G}(U;\mathscr{O}^{(0,n)}_{X\times X}) \longrightarrow 0,
\]
where $\displaystyle W_{\hat{i}} = \bigcap_{j\neq i} W_j$ and $\displaystyle W = \bigcap^n_{j=1} W_j$.
Hence $P\in H^n_{G}(U;\mathscr{O}^{(0,n)}_{X\times X})$ is represented by some holomorphic form $\psi = \psi(z,z'-z)dz' \in \Gamma (W;\mathscr{O}^{(0,n)}_{X\times X})$ such that
\[
P = [\psi(z,z'-z)dz'].
\]
We denote by $H^n(\mathscr{O}^{(0,n)}_{X\times X}; \mathcal{W},\mathcal{W}')$ the \v{C}ech cohomology with respect to the coverings $(\mathcal{W},\mathcal{W}')$.
\par

Next we recall the \v{C}ech-Dolbeault expression of $H^n_{G}(U;\mathscr{O}^{(0,n)}_{X\times X})$ with respect to coverings $(\mathcal{V},\mathcal{V}')$.
By the results in Section 3, we have
\[
H^n_{G}(U;\mathscr{O}^{(0,n)}_{X\times X}) = H^{0,n,n}_{\bar{\vartheta}}(\mathcal{V},\mathcal{V}').
\]
Moreover $P\in H^n_{G}(U;\mathscr{O}^{(0,n)}_{X\times X})$ is represented by some $\omega = (\omega_0,\omega_{01}) \in C^{\infty,(0,n;n)}_{X\times X}(\mathcal{V},\mathcal{V}')$ such that
\[
P = [\omega] = [(\omega_0,\omega_{01})].
\]

\begin{rema}
While in Section 3 we set $\mathcal{V}' = \{ V_0 \}$ and the representative of an element of \v{C}ech-Dolbeault cohomology is represented by the pair $(\omega_1,\omega_{01})$,
in this section we set $\mathcal{V}' = \{ V_1 \}$.
Therefore the index of first term of $\omega = (\omega_0,\omega_{01})$ is different from the one in Section 3.
\end{rema}

Finally by applying $(\mathcal{W},\mathcal{W}')$ to the previous subsection, we have the \v{C}ech-Dolbeault cohomology $\displaystyle H^{0,n,n}_{\bar{\vartheta}}(\mathcal{W},\mathcal{W}')$, which is isomorphic to $H^n_G(U;\mathscr{O}^{(0,n)}_{X\times X})$.

\[
\xymatrix@R=30pt@C=35pt{
  & [\omega] \in H^{0,n,n}_{\bar{\vartheta}}(\mathcal{V},\mathcal{V}') \ar[rd]^{\widetilde{\phi}_2} & \\
P\in H^n_{G}(U;\mathscr{O}^{(0,n)}_{X\times X}) \ar[ru] \ar[rd] & & H^{0,n,n}_{\bar{\vartheta}}(\mathcal{W},\mathcal{W}') .\\
  & [\psi] \in H^n (\mathscr{O}^{(0,n)}_{X\times X}; \mathcal{W},\mathcal{W}') \ar[ru]^{\widetilde{\phi}_1} &
}
\]

We introduce the maps $\widetilde{\phi}_1,\widetilde{\phi}_2$.
Let $\phi_1$ be the inclusion map
\[
\phi_1:\Gamma(W;\mathscr{O}^{(0,n)}_{X\times X})    \hookrightarrow C^n (\mathcal{W},\mathcal{W}'; \mathscr{F}^{0}) \subset \mathscr{F}^n(\mathcal{W},\mathcal{W}'),
\]
where $\displaystyle W = \bigcap^n_{i = 1} W_i$, and define
\[
\phi_2: \mathscr{F}^n(\mathcal{V},\mathcal{V}') \to C^0(\mathcal{W},\mathcal{W}';\mathscr{F}^n) \oplus C^1(\mathcal{W},\mathcal{W}'\mathscr{F}^{n-1}) \subset \mathscr{F}^n(\mathcal{W},\mathcal{W}')
\]
by
\begin{align*}
\phi_2(\omega)_\alpha =
\begin{cases}
\omega_0|_{W_\alpha} \quad &\alpha=0 \\
0 \quad &\alpha =1,\dots,n
\end{cases}
\quad , \quad
\phi_2(\omega)_{\alpha\beta} =
\begin{cases}
\omega_{01}|_{W_{\alpha\beta}} \quad &(\alpha,\beta) = (0,1),\dots,(0,n)\\
-\omega_{01}|_{W_{\alpha\beta}} \quad &(\alpha,\beta) = (1,0),\dots,(n,0)\\
0 \quad &\mathrm{otherwise}
\end{cases}.
\end{align*}
Then $\phi_1$ induces the morphism between cohomologies
\[
\widetilde{\phi}_1: H^n(\mathscr{O}^{(0,n)}_{X\times X};\mathcal{W},\mathcal{W}') \longrightarrow H^{0,n,n}_{\bar{\vartheta}}(\mathcal{W},\mathcal{W}'),
\]
and $\phi_2$ induces the morphism between cohomologies
\[
\widetilde{\phi}_2 : H^{0,n,n}_{\bar{\vartheta}}(\mathcal{V},\mathcal{V}') \longrightarrow H^{0,n,n}_{\bar{\vartheta}}(\mathcal{W},\mathcal{W}').
\]
\begin{theo}[\cite{Suwa}, Theorem 3.11, Theorem 4.5]
The induced morphism $\widetilde{\phi}_1$ and $\widetilde{\phi}_2$ are isomorphisms.
\end{theo}

Now we work in $H^{0,n,n}_{\bar{\vartheta}}(\mathcal{W},\mathcal{W}')$.
Fix $P \in H^n_G(U;\mathscr{O}^{(0,n)}_{X\times X})$.
Then we obtain two representatives
\[
P = [\psi] \in H^n (\mathscr{O}^{(0,n)}_{X\times X}; \mathcal{W},\mathcal{W}'),\quad P = [\omega] \in H^{0,n,n}_{\bar{\vartheta}}(\mathcal{V},\mathcal{V}').
\]
These are clearly equivalent to each other in $H^{0,n,n}_{\bar{\vartheta}}(\mathcal{W},\mathcal{W}')$, and hence there exists $\eta \in \mathscr{F}^{n-1}(\mathcal{W},\mathcal{W}')$ such that
\[
\widetilde{\phi}_2(\omega) - \widetilde{\phi}_1(\psi) = D\eta.
\]

\begin{comment}
\[
\xymatrix@R=20pt@C=15pt{
C^0(\mathcal{W},\mathcal{W}';\mathscr{O}^{(0,n)}_{X\times X}) \ar[d] \ar@{^{(}->}[r] & C^0(\mathcal{W},\mathcal{W}';\mathscr{F}^0) \ar[d]^{\delta} \ar[r]^{\bar{\partial}} & C^0(\mathcal{W},\mathcal{W}';\mathscr{F}^1) \ar[r]^(.7){\bar{\partial}} \ar[d]^{\delta} & \cdots \ar[r]^(.3){\bar{\partial}} & C^0(\mathcal{W},\mathcal{W}';\mathscr{F}^n) \ar[r]^(.7){\bar{\partial}} \ar[d]^{\delta} & \cdots \\
C^1(\mathcal{W},\mathcal{W}';\mathscr{O}^{(0,n)}_{X\times X}) \ar[d] \ar@{^{(}->}[r] & C^1(\mathcal{W},\mathcal{W}';\mathscr{F}^0) \ar[d]^(.4){\delta} \ar[r]^{-\bar{\partial}} & C^1(\mathcal{W},\mathcal{W}';\mathscr{F}^1) \ar[r]^(.7){-\bar{\partial}} \ar[d]^(.4){\delta} & \cdots \ar[r]^(.27){-\bar{\partial}} & C^1(\mathcal{W},\mathcal{W}';\mathscr{F}^n) \ar[d]^(.4){\delta} \ar[r]^(.7){-\bar{\partial}} & \cdots\\
\vdots \ar[d] & \vdots \ar[d]^{\delta} & \vdots \ar[d]^{\delta} &  & \vdots & \\
C^n(\mathcal{W},\mathcal{W}';\mathscr{O}^{(0,n)}_{X\times X}) \ar[d] \ar@{^{(}->}[r] & C^n(\mathcal{W},\mathcal{W}';\mathscr{F}^0) \ar[d]^(.4){\delta} \ar[r]^{(-1)^n\bar{\partial}} & C^n(\mathcal{W},\mathcal{W}';\mathscr{F}^1) \ar[d]^(.4){\delta} \ar[r]^(.75){(-1)^n\bar{\partial}} & \cdots &  & \\
\vdots & \vdots & \vdots &  &  & .\\
}
\]
\end{comment}

Here note that $D = \delta + (-1)^n \bar{\partial}$ and this condition can be described concretely with $\eta = (\eta_0,\eta_1,\dots,\eta_n)$ as follows.

\begin{equation*}
\left\{
\begin{aligned}
\widetilde{\phi}_2(\omega_0) &= \bar{\partial} \eta_0, \\
\widetilde{\phi}_2(\omega_{01}) &= \delta \eta_0 - \bar{\partial} \eta_1, \\
0 &= \delta \eta_1 + \bar{\partial}\eta_2, \\
&\vdots \\
0 &= \delta \eta_{n-2} + (-1)^{n-1} \bar{\partial} \eta_{n-1}, \\
-\widetilde{\phi}_1(\psi) &= \delta\eta_{n-1}.
\end{aligned}
\right.
\end{equation*}

\subsection{The proof of the commutativity of \eqref{com2map}}

Now we are ready to prove the commutativity of \eqref{com2map}.
Let $w=z'-z$ and $(z,w) = (z_1,\dots,z_n,w_1,\dots,w_n)$ be a local coordinate system of $X \times X \subset \mathbb{C}^n\times \mathbb{C}^n$.
We define several canonical projections as follow.
\begin{alignat*}{1}
q_i: & \mathbb{C}_{n+1} \times \cdots\times \mathbb{C}_{2n} \to \mathbb{C}_{n+i}, \\
q_{ij}: & \mathbb{C}_{n+1}\times \cdots\times \mathbb{C}_{2n} \to \mathbb{C}_{n+i} \times \mathbb{C}_{n+j}.
\end{alignat*}

Let $r$ and $\varepsilon$ be the same as those that appeared in \eqref{cohomologyEX}.
We construct the paths $\gamma_1 \subset \mathbb{C}$ and $\gamma_i \subset\mathbb{C}^2\  (i=2,\dots,n)$ of integration as follows.
(For more details, see \cite{A2}.)
\begin{enumerate}
\item Let $\beta_0,\beta_1$ be sufficiently small complex numbers satisfying
\begin{align*}
& 0 < \mathrm{Re}\, \lambda\beta_0 < \varepsilon \mathrm{Im}\, \lambda\beta_0 ,\\
& 0 < \mathrm{Re}\, \lambda\beta_1 < -\varepsilon \mathrm{Im}\, \lambda\beta_1.
\end{align*}
Then we choose $\gamma_1 \subset \mathbb{C}_{n+1}$ as the path that goes counterclockwise around the origin from $\beta_0$ to $\beta_1$.
\item Fix $w_1\in \gamma_1$.
For $i=2,\dots,n$ we take a sufficiently small $\delta>0$ and set
\[
\gamma_{i} = \left\{ w_i \in \mathbb{C}_{n+i} \, \middle| \,|w_i| = \frac{|w_1|}{\varepsilon} + \delta \right\}.
\]
\end{enumerate}

We also define the domains of integration $B_1,B_2,\dots,B_n$.
The domain $B_1$ satisfies the following conditions:
\begin{enumerate}
\item $B_1\subset \mathbb{C}_{n+1}$ is contractible.
\item The boundary $\partial_1 B_1$ contains $\gamma_1$.
Here $\partial_1$ is the boundary operator taken in $\mathbb{C}_{n+1}$.
\item There exists a positive constant $k$ such that
\[
\sup \left\{ \frac{\langle \lambda, w_1 \rangle}{|\lambda||w_1|} \, \middle| \, w_1 \in (\partial_1 B_1 \setminus \gamma_1) \right\} < -k.
\]
\end{enumerate}

For $i=2,3,\dots,n$, set $B_i\subset \mathbb{C}_{n+1} \times \mathbb{C}_{n+i}$ as follows:\par
\begin{enumerate}
\item The domain $B_i$ is contractible.
\item If we fix $w_1\in \gamma_1$ the boundary $\partial_i B_i$ is equal to $\gamma_i$.
Here $\partial_i$ is the boundary operator taken in $\mathbb{C}_{n+i}$.
\end{enumerate}
Moreover we assume that the orientation of $\partial_i B_i$ is the same as that of $\gamma_i$.
In the following we consider $z=(z_1,\dots,z_n)$ as the parameter.
Then the family $\displaystyle \{R_i \}^n_{i=0}$ of closed sets defined below is the honeycomb system adapted to $\mathcal{W}$ with respect to the variable $w$.
\begin{alignat*}{1}
R_0 &= q^{-1}_1(B_1) \cap \left( \bigcap^n_{k = 2} q^{-1}_{1k}(B_k) \right) ,  \quad R_1 = (U\setminus q^{-1}_1(\mathrm{int}\,B_1)) \cap \left( \bigcap^n_{k = 2} q^{-1}_{1k}(B_k) \right) ,  \\
R_i &=  (U\setminus q^{-1}_{1i}(\mathrm{int}\,B_i))\cap \left( \bigcap^n_{k = i+1} q^{-1}_{1k}(B_k) \right) \quad(i = 2,\dots, n-1),\\
R_n &= U\setminus q^{-1}_{1n}(\mathrm{int}\,B_n).
\end{alignat*}

\begin{figure}[H]
\centering
\begin{tikzpicture}

\draw[->] (0,0) -- (-1,-0.2);
\draw (-1,-0.2) node[below]{$\bar{\lambda}$};

\draw[densely dotted] (0,0)--(1.2,3.6);
\draw[densely dotted] (0,0)--(2.3,-2.3);

\draw[densely dotted] (0,0) -- (-3.5,-0.7);

\draw[densely dotted] (0.7,-3.5) -- (-0.8,4);

\draw[densely dotted] (0.2,3) to [out =240 ,in = 75] (-0.2,2);
\draw[blue] (-0.2,2) to [out=255, in = 120]  (0.7,-1.8);
\draw[densely dotted] (0.7,-1.8) to [out = 300 , in = 130] (1.2,-2.5);

\draw (0.2,3) node[above]{$\beta'_1$};
\draw (1.2,-2.5) node[below]{$\beta'_0$};

\fill (0.2,3) circle [radius=1pt];
\fill (1.2,-2.5) circle [radius=1pt];

\draw (-0.26,2.1) node[above]{$\beta_1$};
\draw (0.7,-1.8) node[below]{$\beta_0$};

\fill (-0.2,2) circle [radius=1pt];
\fill (0.7,-1.8) circle [radius=1pt];

\draw (0.2,3) to [out = 330, in = 30] (1.2,-2.5);
\draw (1.2,-2.5) to [out = 210, in = 280] (-3,-0.6);
\draw (-3,-0.6) to [out = 100, in = 150] (0.2,3);
%\draw (3.6,1.3) -- (0.2,-0.3) to [out=200, in=250] (-0.3,0.2) -- (1.3,3.6) -- (0.1,0.25) to [out=240, in=0] (0,0.2) to [out=180, in=135] (-0.1,-0.1) to [out=315, in=270] (0.2,0) to [out=90, in=30] (0.25,0.1) -- (3.6,1.3);

\draw (-0.3,0.3) node[left]{$\gamma_1$};

\end{tikzpicture}
\begin{tikzpicture}

\fill[blue!30] (0.2,3) to [out =240 ,in = 75] (-0.2,2) to [out=255, in = 120]  (0.7,-1.8) to [out = 300 , in = 130] (1.2,-2.5) to [out = 30, in = 330] (0.2,3);

\draw[->] (0,0) -- (-1,-0.2);
\draw (-1,-0.2) node[below]{$\bar{\lambda}$};

\draw[densely dotted] (0,0)--(1.2,3.6);
\draw[densely dotted] (0,0)--(2.3,-2.3);

\draw[densely dotted] (0,0) -- (-3.5,-0.7);

\draw[densely dotted] (0.7,-3.5) -- (-0.8,4);

\draw[densely dotted] (0.2,3) to [out =240 ,in = 75] (-0.2,2);
\draw[blue] (-0.2,2) to [out=255, in = 120]  (0.7,-1.8);
\draw[densely dotted] (0.7,-1.8) to [out = 300 , in = 130] (1.2,-2.5);

\draw (0.2,3) node[above]{$\beta'_1$};
\draw (1.2,-2.5) node[below]{$\beta'_0$};

\fill (0.2,3) circle [radius=1pt];
\fill (1.2,-2.5) circle [radius=1pt];

\draw (-0.26,2.1) node[above]{$\beta_1$};
\draw (0.7,-1.8) node[below]{$\beta_0$};

\fill (-0.2,2) circle [radius=1pt];
\fill (0.7,-1.8) circle [radius=1pt];

\draw (0.2,3) to [out = 330, in = 30] (1.2,-2.5);
\draw (1.2,-2.5) to [out = 210, in = 280] (-3,-0.6);
\draw (-3,-0.6) to [out = 100, in = 150] (0.2,3);
%\draw (3.6,1.3) -- (0.2,-0.3) to [out=200, in=250] (-0.3,0.2) -- (1.3,3.6) -- (0.1,0.25) to [out=240, in=0] (0,0.2) to [out=180, in=135] (-0.1,-0.1) to [out=315, in=270] (0.2,0) to [out=90, in=30] (0.25,0.1) -- (3.6,1.3);

\draw (1.6,1.6) node[right]{$B_1$};

\end{tikzpicture}
\caption{$\gamma_1$ and $B_1$}
\end{figure}

We set
\begin{align*}
I &= \{1,2,\dots,n\} \\
I^{(r)} &= \{ \alpha^{(r)} = (\alpha_1,\dots,\alpha_r)\in I^r \mid \alpha_1 \leq \dots \leq \alpha_r \}.
\end{align*}
For simplicity we adopt the following notations.
Let $k$ be an integer with $0\leq k \leq n$ and $\alpha^{(r)} = (\alpha_1,\dots,\alpha_r) \in I^{(r)}$.
Then $(0,\alpha_1,\dots,\alpha_r)$ is also denoted by $0\alpha^{(r)}$ and $(\alpha_1,\dots,\alpha_r,k)$ is also denoted by $\alpha^{(r)}k$.
Set $\alpha^{(r)}_{\check{i}} = (\alpha_1,\dots,\alpha_{i-1},\alpha_{i+1},\dots,\alpha_r)$.\par

For any $\alpha^{(r)} \in I^{(r)}$ one sets $R_{\alpha^{(r)}}$ by
\[
R_{\alpha^{(r)}} = \underset{i}{\bigcap} R_{i},
\]
where $i$ ranges through all the components of $\alpha^{(r)}$.
The order of the subscript in the honeycomb system means its orientation.
That is, for any permutation $\rho$ we have
\[
R_{\alpha^{(r)}} = \mathrm{sgn}\, \rho \cdot R_{\rho(\alpha^{(r)})}.
\]

\begin{rema}\label{rmkdiff}
We have to be careful when we take the boundary of $R_{0\alpha^{(r)}}$ with respect to the first variable $w_1$.
By the recipe of honeycomb system $\{ R_i \}^n_{i=0}$ it can be easily shown that for any $i\neq 1$
\[
\partial_i R_{0\alpha^{(r)}} = R_{0\alpha^{(r)}i}.
\]
On the other hand if $i=1$ this does not hold because the reminder domain $\partial_1 R_{0\alpha^{(r)}} \setminus  R_{0\alpha^{(r)}1}$ does not vanish.
Here we give the property of $\partial_1 R_{0\alpha^{(r)}} \setminus  R_{0\alpha^{(r)}1}$.
The definition of $B_1$ implies that there exists a positive constant $k$ such that
\[
\sup \left\{ \frac{\langle \lambda, w \rangle}{|\lambda||w|} \, \middle| \, w\in \partial_1 R_{0\alpha^{(r)}} \setminus  R_{0\alpha^{(r)}1} \right\} < -k.
\]
By noticing this fact that we can see that the integration on the remainder domain is in $\mathfrak{N}^\infty$ class.
In fact, if $\alpha^{(r)}$ contains $1$ the boundary of $R_{0\alpha^{(r)}}$ in $\mathbb{C}_1$ consists of two points $\beta'_0,\beta'_1$.
\[
\partial_1 R_{0\alpha^{(r)}} = q^{-1}_1(\{w_1 \mid w_1=\beta'_0-\beta'_1\}) \cap  R_{0\alpha^{(r)}}.
\]
Hence there exists a positive constant $k_1$ such that
\[
\sup \left\{ \frac{\langle \lambda, w \rangle}{|\lambda||w|} \, \middle| \, w \in \partial_1 R_{0\alpha^{(r)}} \right\} < -k_1.
\]
Also if $\alpha^{(r)}$ does not contain $1$ by the same reason there exists a positive constant $k_2$ such that
\[
\sup \left\{ \frac{\langle \lambda, w \rangle}{|\lambda||w|} \, \middle| \, w \in \partial_1 R_{0\alpha^{(r)}} \setminus  R_{0\alpha^{(r)}1} \right\} < -k_2.
\]
\end{rema}

Recall that $p_1$ is the canonical projection
\[
p_1 : \mathbb{C}^n \times \mathbb{C}^n \to \mathbb{C}^n,\quad (z_1,\dots,z_n,w_1,\dots,w_n) \mapsto (z_1,\dots,z_n).
\]
We abbreviate $R_{\alpha^{(r)}} \cap p^{-1}_1(z)$ to $R_{\alpha^{(r)},z}$ for short.

\begin{defi}
Let $A$ and $B$ be two symbols of $C^\infty$ type.
One writes
\[
A \approx B
\]
if and only if $A-B$ is a null symbol of $C^\infty$ type.
\end{defi}

\begin{theo}\label{NisoA}
Let $\eta = (\eta_0,\eta_1,\dots,\eta_{n-1})$ belong to $\mathrm{Ker}\,(\mathscr{F}^{n-1}(\mathcal{W},\mathcal{W}') \overset{D}{\to} \mathscr{F}^n(\mathcal{W},\mathcal{W}'))$.
Then for any $1\leq r \leq n-2$,
\[
\sum_{\alpha^{(r)}\in I^{(r)}} \int _{R_{0\alpha^{(r)},z} } (\bar{\partial} \eta_r)^{0\alpha^{(r)}} \cdot e^{\langle w,\zeta\rangle} \approx - \sum_{\alpha^{(r+1)}\in I^{(r+1)}} \int _{R_{0\alpha^{(r+1)},z}} (\bar{\partial} \eta_{r+1})^{0\alpha^{(r+1)}} \cdot e^{\langle w,\zeta\rangle}.
\]
\end{theo}

\begin{proof}
Since it follows from the assumption that $\delta \eta_r = (-1)^r \bar{\partial}\eta_{r+1}$ we have
\begin{align*}
&\sum_{\alpha^{(r)}\in I^{(r)}} \int _{R_{0\alpha^{(r)},z}} (\bar{\partial} \eta_r)^{0\alpha^{(r)}} \cdot e^{\langle w,\zeta\rangle} \\
= &\sum_{\alpha^{(r)}\in I^{(r)}} \int _{\partial R_{0\alpha^{(r)},z}} \eta_r^{0\alpha^{(r)}} \cdot e^{\langle w,\zeta\rangle} \\
= &\sum_{1\notin \alpha^{(r)}\in I^{(r)}} \int _{\partial R_{0\alpha^{(r)},z}} \eta_r^{0\alpha^{(r)}} \cdot e^{\langle w,\zeta\rangle} + \sum_{1\in \alpha^{(r)}\in I^{(r)}} \int _{\partial R_{0\alpha^{(r)},z}} \eta_r^{0\alpha^{(r)}} \cdot e^{\langle w,\zeta\rangle} \\
= &\sum_{1\notin \alpha^{(r)}\in I^{(r)}} \left( \int_{\partial_1 R_{0\alpha^{(r)},z} \setminus  R_{0\alpha^{(r)}1,z}} \eta_r^{0\alpha^{(r)}} \cdot e^{\langle w,\zeta\rangle} + \sum^n_{j=1} \int _{R_{0\alpha^{(r)}j,z}} \eta_r^{0\alpha^{(r)}} \cdot e^{\langle w,\zeta\rangle} \right) \\
&\hspace{40pt}  + \sum_{1\in \alpha^{(r)}\in I^{(r)}} \left( \int_{\partial_1 R_{0\alpha^{(r)},z} \setminus  R_{0\alpha^{(r)}1,z}} \eta_r^{0\alpha^{(r)}} \cdot e^{\langle w,\zeta\rangle} + \sum^{n}_{j=1}\int _{R_{0\alpha^{(r)}j,z}} \eta_r^{0\alpha^{(r)}} \cdot e^{\langle w,\zeta\rangle}\right).
\end{align*}
By the argument in Remark \ref{rmkdiff}, there exist positive constants $k_1$ and $C_1$ such that
\[
\left| \sum_{1\notin \alpha^{(r)}\in I^{(r)}}  \int_{\partial_1 R_{0\alpha^{(r)},z} \setminus  R_{0\alpha^{(r)}1,z}} \eta_r^{0\alpha^{(r)}} \cdot e^{\langle w,\zeta\rangle}\right| \leq C_1 e^{-k_1|\zeta|}.
\]
Similarly there exist positive constants $k_2$ and $C_2$ such that
\[
\left|\sum_{1\in \alpha^{(r)}\in I^{(r)}} \int_{\partial_1 R_{0\alpha^{(r)},z} \setminus  R_{0\alpha^{(r)}1,z}} \eta_r^{0\alpha^{(r)}} \cdot e^{\langle w,\zeta\rangle}\right| \leq C_2 e^{-k_2|\zeta|}.
\]
These imply the integrations on reminder domains are in null class.
Hence we have
\begin{align*}
&\sum_{\alpha^{(r)}\in I^{(r)}} \int _{R_{0\alpha^{(r)},z}} (\bar{\partial} \eta_r)^{0\alpha^{(r)}} \cdot e^{\langle w,\zeta\rangle} \\
\approx &\sum_{1\notin \alpha^{(r)}\in I^{(r)}}  \sum^n_{j=1} \int _{R_{0\alpha^{(r)}j,z}} \eta_r^{0\alpha^{(r)}} \cdot e^{\langle w,\zeta\rangle} + \sum_{1\in \alpha^{(r)}\in I^{(r)}} \sum^{n}_{j=1}\int _{R_{0\alpha^{(r)}j,z}} \eta_r^{0\alpha^{(r)}} \cdot e^{\langle w,\zeta\rangle} \\
= & \sum_{\alpha^{(r)}\in I^{(r)}} \sum^n_{j=1} \int _{R_{0\alpha^{(r)}j,z}} \eta_r^{0\alpha^{(r)}} \cdot e^{\langle w,\zeta\rangle}.
\end{align*}

Here we prove the following key lemma.
\begin{lemm}\label{combi}
Let $\eta \in \mathscr{F}^{n-1}(\mathcal{W},\mathcal{W}')$.
We have
\[
\sum_{\alpha^{(r)}\in I^{(r)}} \sum^n_{i=1} \int _{R_{0\alpha^{(r)}i,z}} \eta_r^{0\alpha^{(r)}} \cdot e^{\langle w,\zeta\rangle}
= \sum_{\beta^{(r+1)}\in I^{(r+1)}} \sum^{r+1}_{j=1} \int _{R_{0\beta^{(r+1)},z}} (-1)^{r+1-j} \cdot \eta_r^{0\beta^{(r+1)}_{\check{j}}} \cdot e^{\langle w,\zeta\rangle}.
\]
\end{lemm}

\begin{proof}
Let $\{(\alpha^{(r)},i)\mid \alpha^{(r)}\in I^{(r)}, i \in I\}$ and $\{(\beta^{(r+1)},j) \mid \beta^{(r+1)}\in I^{(r+1)}, 1\leq j\leq r+1 \}$ be index sets.
We denote by $\alpha^{(r)}[i]$ the $i$-th component of $\alpha^{(r)}$.
We define the map $F$ by
\[
\begin{array}{ccc}
F:\{(\alpha^{(r)},i)\mid \alpha^{(r)}\in I^{(r)}, i \in I\}   & \longrightarrow & \{(\beta^{(r+1)},j) \mid \beta^{(r+1)}\in I^{(r+1)}, j\in I \} \\[2pt]
\rotatebox{90}{$\in$} &                 & \rotatebox{90}{$\in$} \\[-1pt]
(\alpha^{(r)},i) & \longmapsto & (\gamma^{(r+1)},k),
\end{array}
\]
where $\gamma^{(r+1)}$ is $\alpha^{(r)} i$ sorted in the increasing order and $k = \#\{ \ell \mid \alpha^{(r)}[\ell] < i \}$.
Remark that if $i\in\alpha^{(r)}$ we have $\alpha^{(r)}i = 0$.
We also define $G$ by
\[
\begin{array}{ccc}
G: \{(\beta^{(r+1)},j) \mid \beta^{(r+1)}\in I^{(r+1)}, j\in I \} & \longrightarrow & \{(\alpha^{(r)},i)\mid \alpha^{(r)}\in I^{(r)}, i \in I\} \\[2pt]
\rotatebox{90}{$\in$} &                 & \rotatebox{90}{$\in$} \\[-1pt]
(\beta^{(r+1)},j) & \longmapsto & (\beta^{(r+1)}_{\check{j}},\beta^{(r+1)}[j]).
\end{array}
\]
Set
\begin{alignat*}{1}
p((\alpha^{(r)},i)) &= \int _{R_{0\alpha^{(r)}i,z}} \eta_r^{0\alpha^{(r)}} \cdot e^{\langle w,\zeta\rangle}, \\
q((\beta^{(r+1)},j)) &= \int _{R_{0\beta^{(r+1)},z}} (-1)^{r+1-j} \cdot \eta_r^{0\beta^{(r+1)}_{\check{j}}} \cdot e^{\langle w,\zeta\rangle}.
\end{alignat*}
By the definitions of $F$ and $G$ we have the following properties.
\begin{enumerate}
\item $F\circ G= id$ and $G\circ F = id$. 
\item $q(F((\alpha{(r)},i))) =p((\alpha{(r)},i))$.
\end{enumerate}
The second property can be shown as follows.
\begin{alignat*}{1}
q(F((\alpha{(r)},i))) &=  q((\gamma^{(r+1)},k))\\
&= \int _{R_{0\gamma^{(r+1)},z}} (-1)^{r+1-k} \cdot \eta_r^{0\gamma^{(r+1)}_{\check{k}}} \cdot e^{\langle w,\zeta\rangle} \\
 &= (-1)^{r+1-k} \int _{R_{0\alpha^{(r)}i,z}} (-1)^{r+1-k} \eta_r^{0\alpha^{(r)}} \cdot e^{\langle w,\zeta\rangle} = p((\alpha{(r)},i)).
\end{alignat*}
Hence we have
\begin{alignat*}{1}
&\sum_{\alpha^{(r)}\in I^{(r)}} \sum^n_{i=1} \int _{R_{0\alpha^{(r)}j,z}} \eta_r^{0\alpha^{(r)}} \cdot e^{\langle w,\zeta\rangle} \\
=& \sum_{\alpha^{(r)}\in I^{(r)}} \sum^n_{i=1} p((\alpha^{(r)},i))\\
=& \sum_{\alpha^{(r)}\in I^{(r)}} \sum^n_{i=1} q(F((\alpha^{(r)},i))) \\
=& \sum_{\beta^{(r+1)}\in I^{(r+1)}} \sum^{r+1}_{j=1} q((\beta^{(r+1)},j))\\
=& \sum_{\beta^{(r+1)}\in I^{(r+1)}} \sum^{r+1}_{j=1} \int _{R_{0\beta^{(r+1)},z}} (-1)^{r+1-j} \cdot \eta_r^{0\beta^{(r+1)}_{\check{j}}} \cdot e^{\langle w,\zeta\rangle}.
\end{alignat*}
\end{proof}

Now we go back to the proof of Theorem \ref{NisoA}.
By Lemma \ref{combi} we obtain
\begin{align*}
&\sum_{\alpha^{(r)}\in I^{(r)}} \sum^n_{i=1} \int _{R_{0\alpha^{(r)}i,z}} \eta_r^{0\alpha^{(r)}} \cdot e^{\langle w,\zeta\rangle} \\
= &  \sum_{\alpha^{(r+1)}\in I^{(r+1)}} \sum^{r+1}_{j=1} \int _{R_{0\alpha^{(r+1)},z}} (-1)^{r+1-j} \cdot \eta_r^{0\alpha^{(r+1)}_{\check{j}}} \cdot e^{\langle w,\zeta\rangle} \\
= & (-1)^{r+1}\sum_{\alpha^{(r+1)}\in I^{(r+1)}} \int _{R_{0\alpha^{(r+1)},z}} (\delta \eta_r)^{0\alpha^{(r+1)}} \cdot e^{\langle w,\zeta\rangle} \\
= & (-1)^{r+1}\sum_{\alpha^{(r+1)}\in I^{(r+1)}} \int _{R_{0\alpha^{(r+1)},z}} ((-1)^r\bar{\partial} \eta_{r+1})^{0\alpha^{(r+1)}} \cdot e^{\langle w,\zeta\rangle} \\
= & - \sum_{\alpha^{(r+1)}\in I^{(r+1)}} \int _{R_{0\alpha^{(r+1)},z}} (\bar{\partial} \eta_{r+1})^{0\alpha^{(r+1)}} \cdot e^{\langle w,\zeta\rangle}
\end{align*}
and this completes the proof.
\end{proof}

Mentioning that $\eta^i_0 = 0$ for $i\neq 0$,
we can calculate the image $\varsigma(\omega)$ as follows.
\begin{align*}
\varsigma(\omega) &= \int_{R_{0,z}} \varphi(\omega_1)^0 \cdot e^{\langle w,\zeta\rangle} + \sum^n_{i=1} \int_{R_{0i,z}} \varphi(\omega_{01})^{0i} \cdot e^{\langle w,\zeta\rangle} \\
&= \int_{R_{0,z}} \bar{\partial} \eta^0_0 \cdot e^{\langle w,\zeta\rangle} + \sum^n_{i=1} \int_{R_{0i,z}} ((\delta\eta_0)^{0i} - (\bar{\partial}\eta_1)^{0i}) \cdot e^{\langle w,\zeta\rangle} \\
&\approx \sum^n_{i=1} \int_{R_{0i,z}} \eta^0_0 \cdot e^{\langle w,\zeta\rangle} + \sum^n_{i=1} \int_{R_{0i,z}} (\delta\eta_0)^{0i} \cdot e^{\langle w,\zeta\rangle} - \sum^n_{i=1} \int_{R_{0i,z}} (\bar{\partial}\eta_1)^{0i} \cdot e^{\langle w,\zeta\rangle} \\
&= \sum^n_{i=1} \int_{R_{0i,z}} \eta^0_0 \cdot e^{\langle w,\zeta\rangle} + \sum^n_{i=1} \int_{R_{0i,z}} ( - \eta^0_0) \cdot e^{\langle w,\zeta\rangle} - \sum_{\alpha^{(1)}\in I^{(1)}} \int_{R_{0\alpha^{(1)},z}} (\bar{\partial}\eta_1)^{0\alpha^{(1)}} \cdot e^{\langle w,\zeta\rangle}\\
& = - \sum_{\alpha^{(1)}\in I^{(1)}} \int_{R_{0\alpha^{(1)},z}} (\bar{\partial}\eta_1)^{0\alpha^{(1)}} \cdot e^{\langle w,\zeta\rangle}.
\end{align*}

By applying Theorem \ref{NisoA} to the above inductively we obtain
\begin{align*}
-\sum_{\alpha^{(1)}\in I^{(1)}} \int _{R_{0\alpha^{(1)},z}} (\bar{\partial} \eta_1)^{0\alpha^{(1)}} \cdot e^{\langle w,\zeta\rangle} &\approx \sum_{\alpha^{(2)}\in I^{(2)}} \int _{R_{0\alpha^{(2)},z}} (\bar{\partial} \eta_{2})^{0\alpha^{(r+1)}} \cdot e^{\langle w,\zeta\rangle}\\
& \quad \vdots \\
& \approx  (-1)^{n-1}\sum_{\alpha^{(n-1)}\in I^{(n-1)}} \int _{R_{0\alpha^{(n-1)},z}} (\bar{\partial} \eta_{n-1})^{0\alpha^{(n-1)}} \cdot e^{\langle w,\zeta\rangle}.
\end{align*}
By Stokes formula we have
\begin{align*}
& (-1)^{n-1}\sum_{\alpha^{(n-1)}\in I^{(n-1)}} \int _{R_{0\alpha^{(n-1)},z}} (\bar{\partial} \eta_{n-1})^{0\alpha^{(n-1)}} \cdot e^{\langle w,\zeta\rangle} \\
\approx&  (-1)^{n-1} \sum_{\alpha^{(n-1)}\in I^{(n-1)}} \sum^n_{j=1} \int_{R_{0\alpha^{(n-1)}j,z}} \eta^{0\alpha^{(n-1)}}_{n-1} \cdot e^{\langle w,\zeta\rangle} \\
= &  (-1)^{n-1} \sum_{\alpha^{(n)}\in I^{(n)}} \sum^n_{j=1} \int_{R_{0\alpha^{(n)},z}} (-1)^{n-j} \cdot \eta^{0\alpha^{(n-1)}_{\check{j}}}_{n-1} \cdot e^{\langle w,\zeta\rangle} \\
= & - \sum_{\alpha^{(n)}\in I^{(n)}} \int_{R_{0\alpha^{(n)},z}} (\delta \eta_{n-1})^{0\alpha^{(n)}} \cdot e^{\langle w,\zeta\rangle} \\
= &  - \sum_{\alpha^{(n)}\in I^{(n)}} \int_{R_{0\alpha^{(n)},z}} (-\sigma)^{0\alpha^{(n)}} \cdot e^{\langle w,\zeta\rangle} \\
= &  \int_{R_{01\dots n,z}} \sigma \cdot e^{\langle w,\zeta\rangle} .
\end{align*}
To conclude the proof of the commutativity of \eqref{com2map} we show
\[
\int_{R_{01\dots n,z}} \sigma \cdot e^{\langle w,\zeta\rangle} \approx \int_{\gamma_1\times \gamma_2 \times \cdots \times \gamma_n} \sigma \cdot e^{\langle w,\zeta\rangle}.
\]
By the definition of $B_i$ there exists a positive constant $k$ such that the subset $R_{01\dots n,z}\setminus \gamma_1\times \dots \times \gamma_n$ is contained in $\displaystyle \left\{ w \in\mathbb{C}^{n} \,\middle|\, \frac{\langle \lambda, w \rangle}{|\lambda||w|} < -k \right\}$.
Hence for a large enough $C>0$ we have
\begin{align*}
\left| \int_{R_{01\dots n,z}} \sigma \cdot e^{\langle w,\zeta\rangle} - \int_{\gamma_1\times \gamma_2 \times \cdots \times \gamma_n} \sigma \cdot e^{\langle w,\zeta\rangle} \right| \leq & \int_{R_{01\dots n,z} \setminus \gamma_1\times \gamma_2 \times \cdots \times \gamma_n } \left|\sigma \cdot e^{\langle w,\zeta\rangle}\right| \\
\leq & \int_{R_{01\dots n,z} \setminus \gamma_1\times \gamma_2 \times \cdots \times \gamma_n } |\sigma| \cdot e^{-k|\zeta|}\\
\leq & Ce^{-k|\zeta|},
\end{align*}
and the proof has been completed.

\end{document}